\documentclass[12pt,a4paper]{article}
\usepackage{amsmath,amssymb,amsthm}
\usepackage{color}

\usepackage{enumerate}

\newtheorem{theorem}{Theorem}[section]
\newtheorem{lemma}[theorem]{Lemma}
\newtheorem{corollary}[theorem]{Corollary}
\newtheorem{proposition}[theorem]{Proposition}
\theoremstyle{definition}
\newtheorem{definition}[theorem]{Definition}
\newtheorem{remark}[theorem]{Remark}
\newtheorem{example}[theorem]{Example}

\newcommand{\dd}{{\rm d}}
\newcommand{\dis}{\displaystyle}
\newcommand{\N}{{\mathbb N}}
\newcommand{\R}{{\mathbb R}}
\newcommand{\ab}{[a,b\,]}

\DeclareMathOperator{\SV}{SV}
\DeclareMathOperator{\var}{var}
\DeclareMathOperator{\sem}{semivar}
\DeclareMathOperator{\W}{W}

\numberwithin{equation}{section}

\begin{document}

\title{On functions of bounded semivariation}
\author{
G.~A.~Monteiro\thanks{Mathematical Institute, Academy of Sciences of the Czech Republic,
Prague, Czech Republic (email: gam@math.cas.cz).Supported by RVO: 67985840 and by the Academic Human Resource Program
of the Academy of Sciences of the Czech Republic.}}
\date{}
\maketitle

\begin{abstract}
The concept of bounded variation has been generalized in many ways. In the frame 
of functions taking values in Banach space, the concept of bounded semivariation is a very  
important generalization. The aim of this paper is to provide an accessible summary on this notion, 
to illustrate it with an appropriate body of
examples, and to outline its connection with the integration theory due to Kurzweil.
\end{abstract}

\smallskip

\noindent{2010 {\it Mathematics Subject Classification}: 26A45

\smallskip


\section{Introduction}
Different notions of variation appear when dealing with problems in infinite dimension. 
Among them, the semivariation is very frequent, being commonly found in 
studies involving convolution, Stieltjes type integration, and also in topics related to vector measures. 

Initially called $w$-property, the concept of bounded semivariation for operator-valued functions was introduced in 1936 by M. Gowurin 
in his paper on the Stieltjes integral in Banach space \cite{Gowurin}. Some decades later, 
the Gowurin $w$-property revealed  to be very useful in the 
investigation of integral representations of continuous linear transformations (see \cite{Tu} and \cite{EW}). 

Nowadays, a handful of papers make use of the concept of bounded semivariation. However, in its majority, the results 
on such type of variation are only stated with no proofs or no proper references. Besides that, 
we can observe in the literature a lack of material collecting basic results on such a concept.

In view of this, the purpose of this survey is to summarize the present knowledge on semivariation.   
The presentation does not reflect the chronological order of the discoveries,
but rather attempts to organize results in a logical framework. Moreover, in order to 
make these notes self-contained, most results are presented with a detailed proof and 
some illustrative examples are given. We trust that our citations and bibliography 
sufficiently identify the appropriate antecedent.

This survey includes, besides basic results and properties, also a section dedicated 
to the investigation of the relation between semivariation and non-absolute integrals.

\medskip

First, let us fix some notation. 

Throughout this survey $X$ and $Y$ denote Banach spaces and $L(X,Y)$ stands for the Banach space of bounded
linear operators from $X$ to $Y$. By $\|\cdot\|_X$ and $\|\cdot\|_{L(X,Y)}$ we denote the norm in $X$ and 
the usual operator norm in $L(X,Y)$, respectively. In particular, we write $L(X)=L(X,X)$ and $X^*=L(X,\R)$.

For an arbitrary function $f:\ab\to X$ we set $\|f\|_\infty=\sup_{t\in\ab}\|f(t)\|_X$.

Consider a nondegenerate closed interval $[a,b]$ and denote by $\mathcal D[a,b]$ the set of all finite divisions of $[a,b]$ of the form
\[
   D=\{\alpha_0,\alpha_1,\dots,\alpha_{\nu(D)}\},\quad a\,{=}\,\alpha_0\,{<}\,\alpha_1\,{<}\,\dots\,{<}\,\alpha_{\nu(D)}\,{=}\,b\,,
\]
where $\nu(D)\in\N$ corresponds to the number of subintervals in which $[a,b]$ is divided.

\medskip

With these concepts in hand we are ready to define the semivariation of an operator-valued function.

\smallskip

\begin{definition}\label{def-SV}
Given a function $F:\ab\to L(X,Y)$ and a division $D\in\mathcal D\ab,$ let
\[
   V(F,D,\ab)=
   	\sup\left\{\,\Bigg\|\sum_{j=1}^{\nu(D)} \big[F(\alpha_{j})-F(\alpha_{j-1})\big]\,x_j\Bigg\|_Y\,:\,x_j\in X,\,\|x_j \|_X\le 1\right\}.
\]
The {\it semivariation} of $F$ on $\ab$ is then defined by
\[
    \SV_a^b(F)=\sup\{ V(F,D,\ab)\,:\,D\in\mathcal D\ab\}.
\]
If $\SV_a^b(F)<\infty$, we say that the function $F$ is of bounded semivariation on $\ab$.
The set of all functions $F:\ab\to L(X,Y)$ of bounded 
semivariation on $\ab$ we denote by $SV(\ab,L(X,Y))$.
\end{definition}

\smallskip

If no misunderstanding can arise, we write simply $V(F,D)$ instead of $V(F,D,\ab)$.

\smallskip

\begin{remark}
The concept presented in Definition \ref{def-SV}, called $w$-property in \cite{Gowurin}, is also known as $(\mathcal B)$-variation,  
with respect to the bilinear triple $\mathcal B=(L(X,Y),X,Y)$. For details, see \cite{Sch1} and \cite{FB}.  
The terminology used in this paper is consistent with that found in the book by H\" onig \cite{H} and 
seem to be the most frequent in literature. 
However, we call the readers attention to the fact that the term `semivariation' might also appear with slight  
different formulation - for example, when applied to measure theory or to functions with values in a general Banach space. 
See, for instance, \cite{De},  \cite{Di} or, in the frames of functions with values in 
locally convex spaces, \cite{DV}.
\end{remark}

\smallskip

It is not hard to see that the semivariation is more general than the notion of variation in the sense of Jordan. 
Indeed, for $F:\ab\to L(X,Y)$, we have
\[
\SV_a^b(F)\leq\var_a^b(F)
\]
where $\var_a^b(F)$ stands for the variation of $F$ on $\ab$ and is given by
\[
\var_a^b(F)=\sup\left\{\,\sum_{j=1}^{\nu(D)}\|F(\alpha_j){-}F(\alpha_{j{-}1})\|_{L(X,Y)}\,:\,D\in\mathcal D\ab\,\right\}.
\]

Denoting by $BV(\ab,L(X,Y))$ the set of all functions $F:\ab\to L(X,Y)$ of bounded 
variation on $\ab$ (i.e $\var_a^b(F)<\infty$), clearly, 
\[
BV(\ab,L(X,Y))\subseteq SV(\ab,L(X,Y)).
\]
The relation between these two sets will be analysed in more details in Section 4.

\medskip

The following example of a function of bounded semivariation was inspired by some ideas found in \cite{Thorp}.

\smallskip

\begin{example}\label{ex.1}
Let $\ell_2$ be the Banach space of sequences $x=\{x_n\}_n$ in $\R$ such that the series $\sum_{n=1}^\infty |x_n|^2$ converges, 
equipped with the norm
\[
\|x\|_2=\Big(\sum_{n=1}^\infty |x_n|^2\Big)^{1/2}
\]
Denote by $e_k$, $k\in\N$, the canonical Schauder basis of $\ell_2$, where $e_k$ 
is the sequence whose $k$-th term is $1$ and all other terms are zero.

For each $k\in\N$, consider $y_k\in \ell_2$ given by $y_k=\frac{1}{k}e_k$, 
that is, 
\[
y_k=\{y_n^{(k)}\}_n\mbox{ \ with \ }y_k^{(k)}=\frac{1}{k}\mbox{ \ and \ }y_n^{(k)}=0\mbox{ \ for \ }n\not= k.
\]
Note that the series $\sum_{k=1}^\infty y_k$ converges in $\ell_2$ and denote by $S$ its sum. 

\medskip

Let $F:[0,1]\to L(\R,\ell_2)$ be given by
\[
\big(F(t)\big)\,x=\left\{\begin{array}{cl}
\dis x\sum_{k=1}^ny_k&\mbox{if \ }t\in(\frac{1}{n+1},\frac{1}{n}], ~ n\in\N,
\\[4mm]
x\,S&\mbox{if \ }t=0
\end{array}
\right.
\]
for $t\in[0,1]$ and $x\in\R$.

In order to prove that $F\in SV([0,1],L(\R,\ell_2))$, let us consider $D\in\mathcal D[0,1]$ with $D=\{\alpha_0,\alpha_1,\dots,\alpha_{\nu(D)}\}$.
Put 
\[
n_j=\max\{k\in\N\,:\,k\alpha_j\leq 1\}\mbox{ \ for \ } j=1,\ldots,\nu(D),
\]
and $\Lambda=\{j\,:\,n_j< n_{j-1}\}\subset\{2,\ldots,\nu(D)\}$. 
For $x_j\in\R$, $j=1,\ldots,\nu(D)$ with $|x_j|\leq 1$ we have $F(\alpha_j)x_j=\sum_{k=1}^{n_j}y_k$ and consequently
\begin{align*}
\sum_{j=1}^{\nu(D)}[F(\alpha_j)-F(\alpha_{j-1})]x_j
	&=x_1\,\Big(\sum_{k=1}^{n_1}y_k-S\Big)+\sum_{j=2}^{\nu(D)}x_j\,\Big(\sum_{k=1}^{n_j}y_k-\sum_{m=1}^{n_{j-1}}y_m\Big)
	\\&=-x_1\,\Big(\sum_{k=n_1+1}^{\infty}y_k\Big)-\sum_{j\in\Lambda}x_j\,\Big(\sum_{k=n_j+1}^{n_{j-1}}y_k\Big)
\end{align*}

For $k\in\N$, define
\begin{equation}\label{coef}
\lambda_k=\left\{\begin{array}{cl}
-x_j&\mbox{ \ if \ }n_j< k\leq n_{j-1},\quad j\in\Lambda
\\[4mm]
-x_1&\mbox{ \ if \ }k>n_1
\\[4mm]
0&\mbox{ \ otherwise}
\end{array}
\right..
\end{equation}
Using this sequence and the definition of $y_k$, we can write
\[
\Big\|\sum_{j=1}^{\nu(D)}[F(\alpha_j)-F(\alpha_{j-1})]\,x_j\Big\|_2^2=\Big\|\sum_{k=1}^\infty\lambda_k\,y_k\Big\|_2^2
	=\sum_{k=2}^\infty\Big|\frac{\lambda_k}{k}\Big|^2\leq\sum_{k=2}^\infty\frac{1}{k^2}.
\]
This shows that $\SV_0^1(F)\leq \dis\Big(\sum_{k=2}^\infty\frac{1}{k^2}\Big)^{1/2}=\sqrt{\frac{\pi^2}{6}-1}$

We claim that $\SV_0^1(F)=\sqrt{\frac{\pi^2}{6}-1}$. 
Indeed, fixed an arbitrary $N\in\N$, consider $D_N\in\mathcal D[0,1]$ given by 
\[
D_N=\Big\{0,\frac{1}{N},\frac{1}{N-1},\dots,\frac{1}{2},1\Big\},
\]
Thus, for $x_j\in\R$, $j=1,\ldots,N$ with $|x_j|\leq 1$ we have
\begin{align*}
&\Big\|\sum_{\ell=1}^{N-1}[F(\textstyle\frac{1}{\ell})-F(\frac{1}{\ell+1})]x_\ell+[F(\frac{1}{N})-F(0)]x_N\Big\|_2
	\\&\qquad\qquad
	=\Big\|\sum_{\ell=1}^{N-1}x_\ell y_{\ell+1}+\sum_{k=N+1}^{\infty}x_N y_k\Big\|_2
	=\Big(\sum_{k=2}^\infty\Big|\frac{\tilde{x}_k}{k}\Big|^2\Big)^{1/2}
\end{align*}
where $\tilde{x}_k=x_{k-1}$ if $k=2,\ldots, N,$ and $\tilde{x}_k=x_{N}$ for $k\in\N, \,k>N$. 
Taking the supremum over all possible choices of $x_j\in\R$, $j=1,\ldots,N$ with $|x_j|\leq 1$ we obtain 
\[
V(F,D_N,[0,1])=\Big(\sum_{k=2}^\infty\frac{1}{k^2}\Big)^{1/2},
\]
which proves the claim.
\end{example}

\begin{remark}
In the particular case $X=\R$ the space $SV(\ab,L(\R,Y))$ can be regarded as the space of the 
functions of weak bounded variation, usually denoted by $BW(\ab,Y)$ (c.f. \cite{H1}). This is clear once we recall that the weak variation 
of a function $f:\ab\to Y$ is given by $\W_a^b(f)=\sup\{ W(f,D)\,:\,D\in\mathcal D\ab\}$ where
\[
   W(f,D)=
   	\sup\left\{\,\Big\|\sum_{j=1}^{\nu(D)} \big[f(\alpha_{j})-f(\alpha_{j-1})\big]\,\lambda_j\Big\|_Y\,:\,\lambda_j\in \R,\,|\lambda_j |\le 1\right\}
   	\mbox{ for \ }D\in\mathcal D\ab.
\]

Therefore, we can say that the semivariation of the function $F:[0,1]\to L(\R,\ell_2)$ in the Example \ref{ex.1} coincides with the weak variation 
of $f:[0,1]\to\ell_2$ given by $f(t)=(F(t))1$ for $t\in[0,1]$.
\end{remark}

\section{Semivariation: basic results}
This section summarizes basic properties of the semivariation that are often mentioned without proof in papers which are directly or 
indirectly connected to such a notion. In order to make this work as complete as possible, all the proofs are included. 
Most of the results can be found, for instance, in \cite{H1}, \cite{H} and \cite{Sch-conv}. 

\medskip

We start by noting that $SV(\ab,L(X,Y))$ is a vector space.

\smallskip

\begin{proposition}
Let $F,\,G\in SV(\ab,L(X,Y))$ and $\lambda\in\R$ be given. Then both functions  
$(F+G)$ and $(\lambda\,F)$ are of bounded semivariation on $\ab$, and 
\begin{equation}\label{linear}
\SV_a^b (F+G)\leq \SV_a^b (F) + \SV_a^b (G)\quad\mbox{and}\quad \SV_a^b (\lambda F)=|\lambda|\,\SV_a^b(F).
\end{equation}
\end{proposition}
\begin{proof}
The assertions follow from the fact that the relations
\[
V(F+G,D)\leq V(F,D)+V(G,D)
\quad\mbox{and}\quad
V(\lambda\,F,D)= |\lambda|\,V(F,D)
\]
hold  for every division $D\in\mathcal D\ab$.
\end{proof}

\smallskip

According to \eqref{linear}, $\SV_a^b(\,\cdot\,)$ defines a seminorm on the space of functions of bounded semivariation. 
On the other hand, if we put
\begin{equation}\label{SV-norm}
\|F\|_{SV}=\|F(a)\|_{L(X,Y)}+\SV_a^b(F)\quad\mbox{for \ }F\in SV(\ab,L(X,Y),
\end{equation}
then $SV(\ab,L(X,Y))$ becomes a normed space. This fact is a consequence of 
\eqref{linear} together with the following assertion.

\smallskip

\begin{proposition} Let $F\in SV(\ab,L(X,Y))$. 
Then $\SV_a^b(F)=0$ if and only if $F\equiv C$ for some fixed operator $C\in L(X,Y)$.
\end{proposition}
\begin{proof} Clearly, the semivariation of a constant function is zero. Conversely, assume that $\SV_a^b(F)=0$. Given 
$t\in(a,b]$, if we consider the division $D=\{a,t,b\}$ of $\ab$, for any $x\in X$ with $\|x\|_X\leq 1$ we have
\[
\|F(t)x-F(a)x\|_Y=\big\|[F(t)-F(a)]x+[F(b)-F(t)]0\big\|_Y\leq V(F,D,\ab).
\]
Therefore $[F(t)-F(a)]=0\in L(X,Y)$, that is, 
$F$ is a constant function.
\end{proof}

\smallskip

\begin{remark} It is worth mentioning that in the definition of the norm $\|\cdot\|_{SV}$ we can use the fixed value of 
the function in any point of the interval, that is, taking $c\in\ab$, we can consider
\[
\|F\|_{SV}=\|F(c)\|_{L(X,Y)}+\SV_a^b(F),\quad F\in SV(\ab,L(X,Y)).
\]
The choice of the left-ending point of the interval seems to be the most common in the literature, though. 
Therefore, in this work, we assume the norm in $SV(\ab,L(X,Y))$ as introduced in \eqref{SV-norm}.
\end{remark}

\smallskip

Note that, for $F:\ab\to L(X,Y)$ and $t\in\ab$ we have
\[
\|F(t)\|_{L(X,Y)}\leq \|F(a)\|_{L(X,Y)}+\SV_a^b(F).
\]
Hence, every function $F\in SV(\ab,L(X,Y))$ is bounded and
\[
\|F\|_\infty\leq \|F\|_{SV}.
\]
In view of this, we can say that the topology induced in the space $SV(\ab,L(X,Y))$ 
by the supremum norm is weaker than the one induced by $\|\cdot\|_{SV}$.

\medskip

In the sequel we prove that the space of functions of bounded semivariation is complete when equipped with the norm $\|\cdot\|_{SV}$ 
(c.f. \cite[Proposition~4]{Sch-conv} or \cite[I.3.3]{H1}). To this aim, we will need the following convergence result.

\smallskip

\begin{lemma}\label{Helly}
Let $F:\ab\to L(X,Y)$, a sequence $\{F_n\}_n\subset SV(\ab,L(X,Y))$ and a 
constant $M>0$ be such that
\[
\SV_a^b(F_n)\leq M\quad\mbox{for every \ }n\in\N,
\]
and
\[
\lim_{n\to\infty}\|F_n(t)x-F(t)x\|_Y=0\quad\mbox{for every \ }t\in \ab\mbox{ \ and \ }x\in X.
\]
Then \ $\SV_a^b(F)\leq M$.
\end{lemma}
\begin{proof}
Let $D=\{\alpha_0,\alpha_1,\dots,\alpha_{\nu(D)}\}$ be a division 
of $\ab$ and let $x_j\in X$, $j=1,\ldots,\nu(D)$ with $\|x_j\|_X\leq 1$. 
Note that, for each $n\in\N$, we have
\begin{align}\nonumber
&\Big\|\sum_{j=1}^{\nu(D)} [F(\alpha_{j})-F(\alpha_{j-1})]\,x_j\Big\|_Y
	\\\nonumber&\quad\leq \Big\|\sum_{j=1}^{\nu(D)} [F_n(\alpha_{j})-F_n(\alpha_{j-1})]\,x_j\Big\|_Y
		+\Big\|\sum_{j=1}^{\nu(D)} [F(\alpha_{j})-F_n(\alpha_{j})-F(\alpha_{j-1})+F_n(\alpha_{j-1})]\,x_j\Big\|_Y
	\\\label{iq}&\quad \leq M+\sum_{j=1}^{\nu(D)}\| [F(\alpha_{j})-F_{n}(\alpha_{j})]x_j\|_Y
		+\sum_{j=1}^{\nu(D)}\|[F(\alpha_{j-1})-F_{n}(\alpha_{j-1})]x_j\|_Y
\end{align}

Given $\varepsilon>0$, there is $N_D\in\N$ such that
\[
\|[F(\alpha_{j})-F_{N_D}(\alpha_{j})]x_i\|_Y<\frac{\varepsilon}{2\,\nu(D)}\quad\mbox{for \ } j=0,1,\ldots,\nu(D),
\]
where $i=j, j+1$ (whenever it has a sense). Therefore, taking $n=N_D$ in \eqref{iq} we obtain
\[
\Big\|\sum_{j=1}^{\nu(D)} [F(\alpha_{j})-F(\alpha_{j-1})]\,x_j\Big\|_Y< M+\varepsilon
\]
which implies that
\[
V(F,D)\leq M+\varepsilon.
\]
Since $\varepsilon>0$ is arbitrary, it follows that $V(F,D)\leq M$ for every $D\in\mathcal D\ab$, and consequently $\SV_a^b(F)\leq M$. 
\end{proof}

\smallskip

The previous convergence result usually appears applied to some integration theory
This type of result, often mentioned as Helly-Bray theorem (cf. \cite[Theorem I.5.8]{H} or \cite{Du}), will be study in 
Section 5 in the frames of Kurzweil-Stieltjes integral. 

Now, we are ready to prove the 
completeness of the space $SV(\ab,L(X,Y))$. 

\smallskip

\begin{theorem}
$SV(\ab,L(X,Y))$ is a Banach space with respect to the norm $\|\cdot\|_{SV}$.
\end{theorem}
\begin{proof}
Let $\{F_n\}_n$ be a Cauchy sequence in $SV(\ab,L(X,Y))$ with respect to the norm $\|\cdot\|_{SV}$.
This means that given $\varepsilon>0$ there exists $n_0\in\N$ such that
\begin{equation}\label{Cauchy-seq}
\|F_n(t)-F_m(t)\|_{L(X,Y)}\leq \|F_n-F_m\|_{SV}<\varepsilon, \quad n,\,m\geq n_0 \mbox{ \ and \ }t\in\ab.
\end{equation}
Hence, for each $t\in\ab$, $\{F_n(t)\}_n$ is a Cauchy sequence in $L(X,Y)$ which implies 
that there exists $F(t)\in L(X,Y)$ such that
\[
\lim_{n\to\infty}\|F_n(t)-F(t)\|_{L(X,Y)}=0.
\]
Moreover, due to \eqref{Cauchy-seq}, this convergence is uniform on $\ab$. By the fact that $\{F_n\}_n$ is a Cauchy sequence there 
exists $M>0$ such that $\SV_a^b(F_n)\leq M$ for every $n\in\N$. Therefore, by Lemma \ref{Helly} $F\in SV(\ab,L(X,Y))$.

It remains to show that the convergence is true also in the topology induced by the norm $\|\cdot\|_{SV}$.
To this aim, consider a division $D=\{\alpha_0,\alpha_1,\dots,\alpha_{\nu(D)}\}$ 
of $\ab$ and arbitrary $x_j\in X$, $j=1,\ldots,\nu(D)$ with $\|x_j\|_X\leq 1$. By \eqref{Cauchy-seq}, for $n,\,m\geq n_0$, we have
\[
\Big\|\sum_{j=1}^{\nu(D)} [F_n(\alpha_{j})-F_m(\alpha_{j})-F_n(\alpha_{j-1})+F_m(\alpha_{j-1})]\,x_j\Big\|_Y<\varepsilon.
\]
Thus taking the limit $m\to\infty$ we obtain
\[
\Big\|\sum_{j=1}^{\nu(D)} [F_n(\alpha_{j})-F(\alpha_{j})-F_n(\alpha_{j-1})+F(\alpha_{j-1})]\,x_j\Big\|_Y\leq\varepsilon,
\]
that is, $V((F_{n}-F),D)\leq\varepsilon$, for $n\geq n_0$. Since the division $D\in\mathcal D\ab$ is arbitrary, it follows that 
$\lim_{n\to\infty}\SV(F_n-F)=0$, concluding the proof.
\end{proof}

\smallskip

The following theorem proves that the functions of bounded variation are multipliers for the space $SV(\ab,L(X,Y))$ (see 
\cite[Lemma I.1.11]{H}).

\smallskip

\begin{theorem}
Let $F\in SV([a,b],L(X,Y))$ and $G\in BV([a,b],L(X))$. Consider the function $FG:\ab\to L(X,Y)$ given by $(FG)(t)=F(t)\,G(t)$ for $t\in\ab$. 
Then $F\,G\in SV([a,b],L(X,Y))$ and
\[
\SV_a^b(F\,G)\leq \|F\|_\infty\,\var_a^b(G)+\|G\|_\infty\,\SV_a^b(F).
\]
\end{theorem}
\begin{proof}
Consider a division $D=\{\alpha_0,\alpha_1,\dots,\alpha_{\nu(D)}\}$ of $\ab$ and let $x_j\in X$, $j=1\ldots, \nu(D)$ 
with $\|x_j\|_X\leq 1$. Therefore
\begin{align*}
&\Big\|\sum_{j=1}^{\nu(D)}[(F\,G)(\alpha_j)-(F\,G)(\alpha_{j-1})]\,x_j\Big\|_Y\\
&\quad=\Big\|\sum_{j=1}^{\nu(D)}[F(\alpha_j)\,G(\alpha_j)-F(\alpha_j)\,G(\alpha_{j-1})
	+F(\alpha_j)\, G(\alpha_{j-1})-F(\alpha_{j-1})\, G(\alpha_{j-1})]\,x_j\Big\|_Y\\
&\quad\leq \Big\|\sum_{j=1}^{\nu(D)}F(\alpha_j)[G(\alpha_j)-G(\alpha_{j-1})]\,x_j\Big\|_Y
	+\Big\|\sum_{j=1}^{\nu(D)}[F(\alpha_j)-F(\alpha_{j-1})]\,G(\alpha_{j-1})\,x_j\Big\|_Y\\
&\quad\leq \|F\|_\infty\,\sum_{j=1}^{\nu(D)}\|G(\alpha_j)-G(\alpha_{j-1})\|_{L(X)}
	+\|G\|_\infty\left\|\sum_{j=1}^{\nu(D)}[F(\alpha_j)-F(\alpha_{j-1})]\frac{\,G(\alpha_{j-1})\,x_j}{\|G\|_\infty}\right\|_Y\\
&\quad\leq \|F\|_\infty\,\var_a^b(G)+\|G\|_\infty\,\SV_a^b(F).
\end{align*}
This implies that 
\[
V((FG),D)\leq \|F\|_\infty\,\var_a^b(G)+\|G\|_\infty\,\SV_a^b(F)
\]
for every $D\in\mathcal{D}\ab$, wherefrom the result follows.
\end{proof}

\smallskip

The next theorem presents some algebraic properties of the semivariation.

\smallskip

\begin{theorem}\label{basic}
If $F:\ab\to L(X,Y)$ and $[c,d]\subset\ab$, then 
\[
\SV_c^d(F)\leq \SV_a^b(F).
\]
Moreover,
\begin{equation}\label{add}
\SV_a^b (F)\leq \SV_a^c (F) + \SV_c^b (F)\quad\mbox{for \ }c\in\ab.
\end{equation}
\end{theorem}
\begin{proof}
It is easy to see that, for every division $D$ of $[c,d]$, 
taking $\tilde{D}=D\cup\{a,b\}$, we have $\tilde{D}\in\mathcal D\ab$ and
\[
V(F,D,[c,d])\leq V(F,\tilde{D},\ab)\leq \SV_a^b(F),
\]
Therefore $\SV_c^d(F)\leq \SV_a^b(F)$. 

To prove the superadditivity, given $c\in\ab$ and an arbitray division $D\in\mathcal D\ab$, consider  
$D_1=(D\cap[a,c])\cup\{c\}$ and $D_2=(D\cap[c,b])\cup\{c\}$. Clearly, $D_1$ and $D_2$ are divisions of 
$[a,c]$ and $[c,b]$, respectively. In addition,
\[
V(F,D,\ab)\leq V_a^b(F,D\cup\{c\},\ab)\leq V(F,D_1,[a,c])+V(F,D_2,[c,b]).
\]
Hence
\[
V(F,D,\ab)\leq \SV_a^c(F)+\SV_c^b(F)\quad\mbox{for every \ } D\in\mathcal D\ab,
\]
which leads to the inequality \eqref{add}.
\end{proof}

\smallskip

According to the previous theorem: {\it if $F\in SV(\ab,L(X,Y))$, then  
$F$ is of bounded semivariation on each closed subinterval of $\ab$}. 
As a consequence we have the following.

\smallskip

\begin{corollary}\label{mapping}
Let $F\in SV(\ab,L(X,Y))$ be given. Then
\begin{enumerate}
\item the mapping $t\in\ab\longmapsto \SV_a^t (F)$ is nondecreasing;
\item the mapping $t\in\ab\longmapsto \SV_t^b (F)$ is nonincreasing.
\end{enumerate}
\end{corollary}

\smallskip

Theorem \ref{basic} indicates that, unlike the variation, the semivariation need not be additive with respect to intervals. 
Next example shows that the inequality in \eqref{add} may be strict. 

\smallskip

\begin{example}\label{add-ex}
Let $F:[0,1]\to L(\R,\ell_2)$ be the function given on Example \ref{ex.1}, that is, for $t\in[0,1]$ and $x\in\R$,
\[
\big(F(t)\big)\,x=\left\{\begin{array}{cl}
\dis x\sum_{k=1}^n\,y_k&\mbox{if \ }t\in(\frac{1}{n+1},\frac{1}{n}], ~ n\in\N,
\\[4mm]
x\,S&\mbox{if \ }t=0
\end{array}
\right.
\]
where $y_k=\frac{1}{k}e_k\in \ell_2$ \footnote{For $k\in\N$, $e_k$ denotes an element of the canonical Schauder basis of $\ell_2$.}for $k\in\N$, and $S=\sum_{k=1}^\infty y_k$.

We will prove that
\begin{equation}\label{strict}
\SV_0^1(F)<\SV_0^{\frac{1}{2}}(F)+\SV_{\frac{1}{2}}^1(F).
\end{equation}

First, let us calculate $\SV_0^{\frac{1}{2}}(F)$.

Given a division $D=\{\alpha_0,\alpha_1,\dots,\alpha_{\nu(D)}\}$ 
of $[0,\frac{1}{2}]$, as in Example \ref{ex.1}, put
\[
n_j=\max\{k\in\N\,:\,k\alpha_j\leq 1\}\mbox{ \ for \ } j=1,\ldots,\nu(D),
\]
and $\Lambda=\{j\,:\,n_j< n_{j-1}\}\subset\{2,\ldots,\nu(D)\}$. For $x_j\in\R$, $j=1,\ldots,\nu(D)$ with $|x_j|\leq 1$ we have
\[
\Big\|\sum_{j=1}^{\nu(D)}[F(\alpha_j)-F(\alpha_{j-1})]\,x_j\Big\|_2=\Big\|\sum_{k=1}^\infty\lambda_k\,y_k\Big\|_2
	=\Big(\sum_{k=3}^\infty\Big|\frac{\lambda_k}{k}\Big|^2\Big)^{1/2}\leq\Big(\sum_{k=3}^\infty\frac{1}{k^2}\Big)^{1/2}.
\]
where $\lambda_k$ for $k\in\N$, $k\geq 3$, is given as in \eqref{coef} (note that   
the corresponding $n_j$ satisfies $n_j\geq 2$, $j=1,\ldots, \nu(D)$).
In view of this, it is clear that
\[
\SV_0^{\frac{1}{2}}(F)\leq\Big(\sum_{k=3}^\infty\frac{1}{k^2}\Big)^{1/2}=\sqrt{\frac{\pi^2}{6}-\frac{5}{4}}.
\] 
The equality $\SV_0^{\frac{1}{2}}(F)=\sqrt{\frac{\pi^2}{6}-\frac{5}{4}}$ is a consequence of the fact that 
\[
V(F,D_N,[0,\frac{1}{2}])=\Big(\sum_{k=3}^\infty\frac{1}{k^2}\Big)^{1/2}
\]  
for any division $D_N=\Big\{0,\frac{1}{N},\frac{1}{N-1},\dots,\frac{1}{2}\Big\}$ with $N\in\N$.

On the other hand, it is not hard to see that $\SV_{\frac{1}{2}}^1(F)=\frac{1}{2}$. Indeed, for any division 
$D=\{\alpha_0,\alpha_1,\dots,\alpha_{\nu(D)}\}$ of $[\frac{1}{2},1]$ and for  
any choice of $x_j\in\R$, $j=1,\ldots,\nu(D)$ with $|x_j|\leq 1$ we have
\[
\sum_{j=1}^{\nu(D)}[F(\alpha_j)-F(\alpha_{j-1})]\,x_j=[F(\alpha_1)-F(\frac{1}{2})]\,x_1=-x_1\,y_2=-\frac{x_1}{2}\,e_2
\]
Hence, $V(F,D,[\frac{1}{2},1])=\frac{1}{2}$.  

Recalling that $\SV_0^1(F)=\sqrt{\frac{\pi^2}{6}-1}$, we conclude that \eqref{strict} holds.
\end{example}

\smallskip

In the sequel we provide some further characterizations of the semivariation of a function. The first 
one, Theorem \ref{B*}, can be found for instance in \cite[Proposition~1.1]{Sch2} or \cite[Theorem~I.4.4]{H}. Basically, it connects 
the notions of semivariation and $\mathcal B^*$-variation, with respect to the bilinear triple 
$\mathcal B^*=(L(X,Y),L(X),L(X,Y))$ (for definition see \cite{H}).

\smallskip

\begin{theorem}\label{B*}
For $F:\ab\to L(X,Y)$ and $D\in\mathcal D\ab$ put
\[
   V^*(F,D)=
   	\sup\left\{\,\Big\|\sum_{j=1}^{\nu(D)} \big[F(\alpha_{j})-F(\alpha_{j-1})\big]\,G_j\Big\|_{L(X,Y)}\,:\,G_j\in L(X),\,\|G_j \|_{L(X)}\le 1\right\}.
\]
Then 
\[
\SV_a^b(F)=\sup\{ V^*(F,D)\,:\,D\in\mathcal D\ab\}.
\]
\end{theorem}
\begin{proof}
It is enough to show that  
\[
V^*(F,D)=V(F,D)\quad \mbox{for every \ }D\in\mathcal D\ab.
\]

Let $D=\{\alpha_0,\alpha_1,\dots,\alpha_{\nu(D)}\}$ be a division of $\ab$. For  
$G_j\in L(X)$, $j=1,\ldots,\nu(D)$ with $\|G_j\|_{L(X)}\le 1$ we have
\[
\begin{split}
\Big\|\sum_{j=1}^{\nu(D)} \big[F(\alpha_{j})-F(\alpha_{j-1})\big]\,G_j\Big\|_{L(X,Y)}&=
\sup_{\|z\|_X\le 1}\Big\|\Big(\sum_{j=1}^{\nu(D)} \big[F(\alpha_{j})-F(\alpha_{j-1})\big]\,G_j\Big)\,z\Big\|_Y
\\&
=\sup_{\|z\|_X\le 1}\Big\|\sum_{j=1}^{\nu(D)} \big[F(\alpha_{j})-F(\alpha_{j-1})\big]\,(G_j\,z)\Big\|_Y
\\&
\le\sup_{\|y_j\|_X\le 1}\Big\|\sum_{j=1}^{\nu(D)} \big[F(\alpha_{j})-F(\alpha_{j-1})\big]\,y_j\Big\|_Y
\end{split}
\]
(where the last inequality is due to the fact that $\|G_j\,z\|_X\le 1$ provided $\|z\|_X\le 1$). Hence $V^*(F,D)\le V(F,D)$.

To obtain the reversed inequality, let us choose $w\in X$ and $\varphi\in X^*$ such that $\|w\|_X=1$, $\|\varphi\|_{X^*}=1$ and $\varphi(w)=1$ 
(which exists by the Hahn-Banach theorem, c.f. \cite[Theorem 2.7.4]{HP}). 
Given $x_j\in X$, $j=1,\ldots,\nu(D)$ with $\|x_j\|_X\le 1$ consider $G_j\in L(X)$ defined by
\[
G_j\,x=\varphi(x)\,x_j\quad\mbox{for \ } x\in X.
\]
Note that, $\|G_j\|_{L(X)}\le 1$ and $G_j\,w=x_j$ for $j=1,\ldots,\nu(D)$. Thus
\[
\begin{split}
\Big\|\sum_{j=1}^{\nu(D)} \big[F(\alpha_{j})-F(\alpha_{j-1})\big]\,x_j\Big\|_Y&=
\Big\|\sum_{j=1}^{\nu(D)} \big[F(\alpha_{j})-F(\alpha_{j-1})\big]\,(G_j\,w)\Big\|_Y
\\&
=\Big\|\Big(\sum_{j=1}^{\nu(D)} \big[F(\alpha_{j})-F(\alpha_{j-1})\big]\,G_j\Big)\,w\Big\|_Y
\\&
\le\Big\|\sum_{j=1}^{\nu(D)} \big[F(\alpha_{j})-F(\alpha_{j-1})\big]\,G_j\Big\|_{L(X,Y)}\|w\|_X
\end{split}
\]
which yields $V(F,D)\leq V^*(F,D)$ concluding the proof.
\end{proof}

\smallskip

\begin{remark}
In a more general formulation, $V^*(F,D)$ in Theorem \ref{B*} can be defined so that the supremum is taken over 
all possible choices of $G_j\in L(Z,X)$ with $\|G_j \|_{L(Z,X)}\le 1$, $j=1,\ldots,\nu(D)$; where $Z$ is an arbitrary Banach space.
\end{remark}

\smallskip

The following theorem, stated in \cite[3.6, Chapter I]{H}, will be useful for the investigation of continuity type results for semivariation. 
The characterization presented involves functions $(y^*\circ F):\ab\to X^*$, obtained by the composition of $F:\ab\to L(X,Y)$ and a 
functional $y^*\in Y^*$ which reads as follows
\begin{equation}\label{dual-circ}
(y^*\circ F)(t)(x)=y^*\big(F(t)\,x\big)\quad\mbox{for \ } t\in\ab, ~ x\in X.
\end{equation}

\smallskip

\begin{theorem}\label{SV-via-X*}
The semivariation of a function $F:\ab\to L(X,Y)$ is given by
\begin{equation}\label{eq.SV-X*}
   \SV_a^b(F)=\sup\big\{\,\var_a^b\,(y^*\circ F)\,:\,y^*\in Y^*,\,\|y^*\|_{Y^*}\le 1\big\}.
\end{equation}
Moreover, $F\in SV(\ab,L(X,Y))$ if and only if $(y^*\circ F)\in BV(\ab,X^*)$ for all $y^*\in Y^*$.
\end{theorem}
\begin{proof}
Let $D\in\mathcal D\ab$, $D=\{\alpha_0,\alpha_1,\dots,\alpha_{\nu(D)}\}$, be given. For $x_j\in X$, $j=1,\ldots,\nu(D)$ with $\|x_j\|_{X}\le 1$ we have
\begin{align*}
&\Big\|\sum_{j=1}^{\nu(D)}[F(\alpha_j)-F(\alpha_{j-1})]\,x_j\Big\|_Y
\\&\quad
=\sup\left\{\,\Big|y^*\Big(\sum_{j=1}^{\nu(D)}[F(\alpha_j)-F(\alpha_{j-1})]\,x_j\Big)\Big|\,:y^*\in Y^*,\,\|y^*\|_{Y^*}\le 1\,\right\}
\\&\quad
\leq\sup\left\{\,\sum_{j=1}^{\nu(D)}\big|[y^*\circ F(\alpha_j)-y^*\circ F(\alpha_{j-1})]\,x_j\big|\,:\,y^*\in Y^*,\,\|y^*\|_{Y^*}\le 1\,\right\}.
\end{align*}
Therefore,
\[
V(F,D)
\leq
\sup\left\{\,\sum_{j=1}^{\nu(D)}\|y^*\circ F(\alpha_j)-y^*\circ F(\alpha_{j-1})\|_{X^*}\,:\,y^*\in Y^*,\,\|y^*\|_{Y^*}\le 1\,\right\}
\]
for every $D\in\mathcal D\ab$, and consequently
\begin{equation}\label{eq1}
\SV_a^b(F)\leq\sup\big\{\,\var_a^b\,(y^*\circ F)\,:\,y^*\in Y^*,\,\|y^*\|_{Y^*}\le 1\big\}.
\end{equation}

On the other hand, given $y^*\in Y^*$ with $\|y^*\|_{Y^*}\leq 1$ and $\varepsilon> 0$, for $j=1,\ldots,\nu(D)$, there exists $x_j\in X$ 
with $\|x_j\|_X\leq 1$ such that
\[
\|y^*\circ F(\alpha_j)-y^*\circ F(\alpha_{j-1})\|_{X^*}-\frac{\varepsilon}{\nu(D)} 
	< \big|[y^*\circ F(\alpha_j)-y^*\circ F(\alpha_{j-1})]\,x_j\big|.
\]
If we put $\lambda_j:={\rm sgn}\,\big([y^*\circ F(\alpha_j)-y^*\circ F(\alpha_{j-1})]\,x_j\big)$ and $\tilde{x}_j=\lambda_j\,x_j$  
for $j=1,\ldots,\nu(D)$, then we obtain
\begin{align*}
&\sum_{j=1}^{\nu(D)}\|y^*\circ F(\alpha_j)-y^*\circ F(\alpha_{j-1})\|_{X^*}-\varepsilon
\\& \qquad\leq\sum_{j=1}^{\nu(D)}[y^*\circ F(\alpha_j)-y^*\circ F(\alpha_{j-1})]\,\tilde{x}_j
\\& \qquad\leq\Big|\sum_{j=1}^{\nu(D)}[y^*\circ F(\alpha_j)-y^*\circ F(\alpha_{j-1})]\,\tilde{x}_j\Big|
=\Big|y^*\Big(\sum_{j=1}^{\nu(D)}[F(\alpha_j)-F(\alpha_{j-1})]\,\tilde{x}_j\Big)\Big|
\\& \qquad\leq\|y^*\|_{Y^*}\,\Big\|\sum_{j=1}^{\nu(D)}[F(\alpha_j)-F(\alpha_{j-1})]\,\tilde{x}_j\Big\|_Y\leq\SV_a^b(F)
\end{align*}
Taking the surpremum over all $D\in\mathcal D\ab$, we get $\var_a^b(y^*\circ F)\leq \varepsilon+\SV_a^b(F).$ 
Since $\varepsilon>0$ is arbitrary, it follows that 
\[
\var_a^b(y^*\circ F)\leq \SV_a^b(F)\quad\mbox{for $y^*\in Y^*$ with $\|y^*\|_{Y^*}\leq 1$,}
\]
which, together with \eqref{eq1}, proves the result.
\end{proof}

\smallskip

The equality in \eqref{eq.SV-X*} is used, in a more general way, to define the notion of semivariation in the 
frame of functions with values in an arbitrary Banach space (cf. \cite{BGC}). 
More precisely, if $Z$ is a Banach space, the semivariation of 
$f:\ab\to Z$ is given by
\[
\sem_a^b(f)=\sup\big\{\,\var_a^b\,(z^*	\circ f)\,:\,z^*\in Z^*,\,\|z^*\|_{z^*}\le 1\big\}.
\]
where the functions $(z^*\circ f):\ab\to\R$ are defined as in \eqref{dual-circ} with an obvious adaptation.

Thereafter, for operator-valued functions two notions of semivariation can be derived. However, 
no direct connection between them is established since such connection would rely on a characterization of the dual space of $L(X,Y)$. 
On the other hand, as observed in \cite{BGC}, given a function $F:\ab\to L(X,Y)$, those two notions are related as follows: 
for each $x\in X$ the function $F_x: t\in\ab\mapsto F(t)x\in Y$ satisfies
\[
\sem_a^b(F_x)\leq \SV_a^b(F).
\]

\section{Semivariation and variation}
We have mentioned in Section 1 that every function of bounded variation is also of bounded semivariation, that is, 
\begin{equation}\label{BV-SV}
BV(\ab,L(X,Y))\subseteq SV(\ab,L(X,Y)).
\end{equation}
This section is devoted to the study of conditions ensuring the equality of these two sets. 
To start, we investigate the case when $Y$ is the real line.

\smallskip

\begin{theorem}\label{theo-X*}
Let $F:\ab\to X^*$ be given. Then, $F\in SV(\ab, X^*)$ if and only if $F\in BV(\ab, X^*)$. In this case, $\SV_a^b(F)=\var_a^b(F)$.
\end{theorem}
\begin{proof} 
is analogous to the proof of Theorem \ref{SV-via-X*}. In summary, it is a consequence 
of the fact that we can write
\begin{align*}
V(F,D)&=\sup\left\{\,\Big|\sum_{j=1}^{\nu(D)} \big[F(\alpha_{j})-F(\alpha_{j-1})\big]\,x_j\Big|\,:\,x_j\in X,\,\|x_j \|_X\le 1\right\}
\\&=\sum_{j=1}^{\nu(D)} \big\|F(\alpha_{j})-F(\alpha_{j-1})\big\|_{X*}
\end{align*}
for every division $D=\{\alpha_0,\alpha_1,\dots,\alpha_{\nu(D)}\}$ of $\ab$.
\end{proof}

\smallskip

\begin{remark}\label{finite-dim}
Given $n\in\N$ consider a function $F:\ab\to L(X,\R^n)$. Writing $F=(F_1,\ldots,F_n)$ with $F_j:\ab\to X^*$, $j=1,\ldots,n$, it is clear that
\[
\mbox{$F\in BV(\ab, L(X,\R^n))$ if and only if $F_j\in BV(\ab, X^*)$, $j=1,\ldots,n$}
\]
and similarly
\[
\mbox{$F\in SV(\ab, L(X,\R^n))$ if and only if $F_j\in SV(\ab, X^*)$, $j=1,\ldots,n$.}
\]
With this in mind, the assertion in Theorem \ref{theo-X*} can be extended to the case when $Y$ is an Euclidean space. 
More generally: {\it if $Y$ is a finite dimensional Banach space, we have 
\[
\mbox{$F\in SV(\ab, L(X,Y))$ if and only if $F\in BV(\ab, L(X,Y))$.}
\]}
\end{remark}

\smallskip

The following example presents a function of bounded semivariation whose variation is not finite.

\smallskip

\begin{example}\label{ex.3}
Let $F:[0,1]\to L(\R,\ell_2)$ be the function given on Example \ref{ex.1}, that is, for $t\in[0,1]$ and $x\in\R$,
\begin{equation}\label{ex-F}
\big(F(t)\big)\,x=\left\{\begin{array}{cl}
\dis x\sum_{k=1}^n\,y_k&\mbox{if \ }t\in(\frac{1}{n+1},\frac{1}{n}], ~ n\in\N,
\\[4mm]
x\,S&\mbox{if \ }t=0
\end{array}
\right.
\end{equation}
where $y_k=\frac{1}{k}e_k\in \ell_2$ \footnote{For $k\in\N$, $e_k$ denotes an element of the canonical Schauder basis of $\ell_2$.}

We know that $F\in SV([0,1],L(\R,\ell_2))$. On the other hand, since
\[
\|F(\textstyle\frac{1}{k})-F(\frac{1}{k+1})\|_{ L(\R,\ell_2)}=\|y_{k+1}\|_2=\dis\frac{1}{k+1} \mbox{ \ for every \ } k\in\N,
\]
we have
\[
\dis\sum_{k=1}^{N+1}\frac{1}{k}\leq
\sum_{k=1}^N\|\textstyle F(\frac{1}{k})-F(\frac{1}{k+1})\|_{ L(\R,\ell_2)}+\|F(\frac{1}{N+1})-F(0)\|_{L(\R,\ell_2)}\leq\var_0^1(F),
\]
for any choice of $N\in\N$. Therefore $\var_0^1(F)=\infty$.
\end{example}

\smallskip

The main tool for the construction of the function in the example above was the sequence $\{y_n\}_n$ in $\ell_2$ whose series converges but not absolutely. 
Recalling that for infinite dimensional Banach spaces we can always find a sequence with such property 
(due to Dvoretzky-Rogers Theorem \ref{Theo-DR} presented in the appendix), 
one can see that finite dimension is a necessary and sufficient condition for the equivalence between 
bounded variation and bounded semivariation. 
We remark that in \cite[Theorem 2]{Thorp} this equivalence was actually proved for functions defined on a ring of sets. 

Using the ideas from \cite{Thorp}, we will show that for infinite dimensional spaces $Y$ the inclusion in 
\eqref{BV-SV} is strict.

\smallskip   

\begin{theorem}\label{infinite-dim}
If the dimension of $Y$ is infinite, then there exists $F\in SV(\ab,L(X,Y))$ such thar $\var_a^b(F)=\infty$.
\end{theorem}
\begin{proof}
By the Dvoretzky-Rogers Theorem \ref{Theo-DR} and its Corollary \ref{D-R} in the Appendix, 
there exists a sequence $\{y_n\}_n$ in $Y$ such that the series $\sum_{n=1}^\infty y_n$ is unconditionally convergent but not 
absolutely convergent. Considering an increasing sequence $\{t_n\}_n$ in $(a,b)$ converging to $b$ and fixing an arbitrary $\varphi\in X^*$, 
with $\|\varphi\|_{X^*}=1$, let
\[
F(t)\,x=\sum_{t_k<t}\varphi(x)\,y_k\quad\mbox{for \ } x\in X \mbox{ \ and \ }t\in\ab.
\]
Note that $F:\ab\to L(X,Y)$ is well-defined (see Theorem \ref{T-1.3.2} in the Appendix).

We claim that the variation of $F$ is not finite. Indeed, given $N\in\N$ consider the division 
$D_N=\{t_0,t_1,\ldots,t_{N+1}, b\}$ formed by elements of the sequence $\{t_n\}_n$ and $t_0=a$. Noting that
\[
\|y_k\|_Y=\|F(t_{k+1})-F(t_k)\|_{L(X,Y)}\quad\mbox{for every \ }k\in\N,
\]
we have
\[
\sum_{k=1}^{N}\|y_k\|_Y\leq\sum_{j=1}^{N+1}\|F(t_j)-F(t_{j-1})\|_{L(X,Y)}+\|F(b)-F(t_{N+1})\|_{L(X,Y)}\leq\var_a^b(F).
\]
Since $\sum_{n=1}^\infty y_n$ is not absolutely convergent, it follows that $\var_a^b(F)=\infty$.

\smallskip

Let us show that $F\in SV(\ab,L(X,Y))$. Consider a division $D=\{\alpha_0,\ldots,\alpha_{\nu(D)}\}$ 
of $\ab$ and let $x_j\in X$, $j=1,\ldots,\nu(D)$ with $\|x_j\|_X\leq 1$. Thus
\[
\Big\|\sum_{j=1}^{\nu(D)}[F(\alpha_j)-F(\alpha_{j-1})]\,x_j\Big\|_Y
=\Big\|\sum_{j=1}^{\nu(D)}\big(\varphi(x_j)\sum_{t_k{\in}[\alpha_{j-1},\alpha_j)}y_k\big)\Big\|_Y
=\Big\|\sum_{k=1}^\infty\beta_k y_k\Big\|_Y
\]
where $\beta_k= \varphi(x_j)$ if $t_k\in[\alpha_{j-1},\alpha_j)$, for some $j=1,\ldots,\nu(D)$, otherwise $\beta_k=0$. 
By Lemma \ref{Thorp-lemma} from the Appendix we know that the set
\[
\left\{\sum_{n=1}^\infty \lambda_n\,y_n\,:\,\lambda_n\in\R\mbox{ \ with \ }|\lambda_n|\leq 1,\,\,n\in\N\,\right\}
\]
is bounded in $Y$. Thus, $\Big\|\sum_{k=1}^\infty\beta_k y_k\Big\|_Y$ is bounded (uniformly with respect to the choice of $x_j\in X$,
 $j=1,\ldots,\nu(D)$) and, consequently, $\SV_a^b(F)<\infty$ which proves the result.
\end{proof}

\smallskip

According to Remark \ref{finite-dim} and Theorem \ref{infinite-dim} we conclude that the notion of 
semivariation is relevant only in spaces with infinite dimension.

\smallskip

\begin{corollary}
The following assertions are equivalent:
\begin{enumerate}[$(i)$]
\item the dimension of $Y$ is finite;
\item every function $F\in SV(\ab,L(X,Y))$ is of bounded variation on $\ab$.
\end{enumerate}
\end{corollary}

\smallskip

Regarding the function $F$ in \eqref{ex-F}, it was shown on Example \ref{add-ex} that 
its semivariation is not additive with respect to intervals (see \eqref{strict}). 
It turns out that such additivity type property can be used to identify whether a function of 
bounded semivariation has a bounded variation as well. This is the content of the following theorem.

\smallskip
 
\begin{theorem}
Let $F\in SV(\ab,L(X,Y))$. Then $F\in BV(\ab,L(X,Y))$ if and only if 
\begin{equation}\label{bv-2}
M:=\sup\left\{\sum_{j=1}^{\nu(D)} \SV_{\alpha_{j-1}}^{\alpha_j}(F)\,:\,D\in{\mathcal D}\ab\right\}<\infty.
\end{equation}
Moreover, in this case, $\var_a^b(F)=M$.
\end{theorem}
\begin{proof} Assume \eqref{bv-2} holds. Given $\varepsilon>0$,  
for $D\in \mathcal D\ab$, with $D=\{\alpha_0,\alpha_1,\dots,\alpha_{\nu(D)}\}$, we can choose $x_j\in X$, $j=1,\ldots,\nu(D)$ with 
$\|x_j\|_X\leq 1$ such that
\[
\|F(\alpha_j)-F(\alpha_{j-1})\|_{L(X,Y)}-\frac{\varepsilon}{\nu(D)}<\|\,[F(\alpha_j)-F(\alpha_{j-1})]\,x_j\|_Y.
\]
Noting that, for $j=1,\ldots,\nu(D)$,
\[
\|\,[F(\alpha_j)-F(\alpha_{j-1})]\,x_j\|_Y\leq \SV_{\alpha_{j-1}}^{\alpha_j}(F),
\]
it follows that
\[
\sum_{j=1}^{\nu(D)}\|F(\alpha_j)-F(\alpha_{j-1})\|_{L(X,Y)}-\varepsilon<\sum_{j=1}^{\nu(D)} \SV_{\alpha_{j-1}}^{\alpha_j}(F)\leq M.
\]
Therefore, taking the supremum over all divisions $D\in\mathcal D\ab$ we obtain
\[
\var_a^b(F)<M+\varepsilon.
\]
Consequently $F\in BV(\ab,L(X,Y))$ and, since $\varepsilon>0$ is arbitrary, $\var_a^b(F)\leq M$. 

On the other hand, for any division $D=\{\alpha_0,\alpha_1,\dots,\alpha_{\nu(D)}\}$ of $\ab$ we have
\[
\sum_{j=1}^{\nu(D)} \SV_{\alpha_{j-1}}^{\alpha_j}(F)\leq \sum_{j=1}^{\nu(D)} \var_{\alpha_{j-1}}^{\alpha_j}(F) = \var_a^b(F),
\]
wherefrom we conclude that $\var_a^b(F)=M$.
\end{proof}

\section{Semivariation: limits and continuity}
It is well-known that a function of bounded variation is regulated, that is, 
the one-sided limits exist at every point of the domain (see \cite[Theorem~I.2.7]{H1} or \cite[Lemma~2.1]{FB}). In this section we investigate 
the connection between functions of bounded semivariation and regulated functions. 

\medskip

Following the notation in \cite{H}, if $f:\ab\to X$ is a regulated function $\ab$, 
we write $f\in G(\ab,X)$, and the one-sided limits are denoted by
\[
   f(t-)=\lim_{s\to t-}f(s)\quad\mbox{and}\quad f(t+)=\lim_{s\to t+}f(s)
\]
for $t\in \ab$ with the convention $f(a-)=f(a)$ and $f(b+)=f(b)$.

\medskip

Another useful notion through this section is the semivariation on half-closed intervals.

\smallskip

\begin{definition}
Given $F:\ab\to L(X,Y)$ and $c,d\in\ab$, $c<d$, the semivariation of $F$ on a half-closed interval $[c,d)$ is given by
\[
\SV_{[c,d)}(F)=\lim_{t\to d-}\SV_c^t(F)=\sup_{t\in[c,d)}\SV_c^t(F).
\]
In analogous way, we define the semivariation on the half-closed interval $(c,d]$ by
\[
\SV_{(c,d]}(F)=\lim_{t\to c+}\SV_t^d(F)=\sup_{t\in(c,d]}\SV_t^d(F).
\]
\end{definition}

\smallskip

Theorems \ref{basic} and \ref{mapping} guarantee that the semivariation over half-closed subintervals of $\ab$ 
is finite for every function from $SV(\ab,L(X,Y))$. 

In what follows we show that a function of bounded semivariation is regulated provided some conditions on 
the semivariation over half-closed intervals are satisfied.

\smallskip

\begin{theorem}\label{regulated}
Let $F\in SV(\ab,L(X,Y))$ be such that
\begin{subequations}
\begin{align}
\lim_{\delta\to 0+} \SV_{[t-\delta,t)}(F)&=0\quad\mbox{for every \ }t\in (a,b],\label{limit-1}
\\
\lim_{\delta\to 0+} \SV_{(t,t+\delta]}(F)&=0 \quad\mbox{for every \ }t\in [a,b).\label{limit-2}
\end{align}
\end{subequations}
Then $F$ is a regulated function on $\ab$.
\end{theorem}
\begin{proof}
Given $t\in(a,b]$ we will prove that $F(t-)\in L(X,Y)$ exists. 
To this aim, consider an increasing sequence $\{t_n\}_n$ in $(a,t)$ converging to $t$.

Let $\varepsilon>0$ be given. By \eqref{limit-1} there exists $\delta>0$ such that
\[
\SV_{[t-\delta,t)}(F)<\varepsilon
\]
Moreover, there is $N\in\N$ so that $t_n>t-\delta$ for every $n\geq N$. Thus, for $m>n>N$ and $x\in X$ with $\|x\|_X\leq 1$ we obtain
\[
\|[F(t_m)-F(t_n)]x\|_Y\leq \SV_{t-\delta}^{t_m}(F)\leq \SV_{[t-\delta,t)}(F)<\varepsilon
\]
which implies that $F(t-)$ exists. Analogously, using \eqref{limit-2}, we can show the existence of $F(t+)$ for every $t\in[a,b)$.
\end{proof}

\smallskip

\begin{remark}
It is not hard to see that, replacing \eqref{limit-1} and \eqref{limit-2} by 
\[
\lim_{\delta\to 0+} \SV_{t-\delta}^t(F)=0\quad\mbox{and}\quad \lim_{\delta\to 0+} \SV_t^{t+\delta}(F)=0
\mbox{ \ for every \ }t\in\ab,
\]
it follows that $F$ is continuous on $\ab$.
\end{remark}

\smallskip

Next lemma is the analogue of \cite[Proposition 4.13]{H1} and provides a condition ensuring that \eqref{limit-1} and \eqref{limit-2} hold.

\smallskip

\begin{lemma}\label{WSC}
If $Y$ is a weakly sequentially complete Banach space, then \eqref{limit-1} and \eqref{limit-2} are satisfied for every $F\in SV(\ab,L(X,Y))$.
\end{lemma}
\begin{proof}
By contradiction assume that there exists a function $F\in SV(\ab,L(X,Y))$ such that for some $t\in (a,b]$ we have 
$\lim_{\delta\to 0+} \SV_{[t-\delta,t)}(F)=M>0$. Hence, there is $\delta_1>0$ such that 
\[
\sup_{s\in [t-\delta,t)}\SV_{t-\delta}^s(F)=\SV_{[t-\delta,t)}(F)>\frac{M}{2}\quad \mbox{for \ }0<\delta\leq \delta_1 .
\]
Put $s_1=t-\delta_1$. In view of the inequality above, there exists $s_2\in(s_1,t)$ so that
\[
\SV_{s_1}^{s_2}(F)>\frac{M}{2}.
\]
Moreover, $\SV_{[s_2,t)}(F)>\frac{M}{2}$. Thus, we can choose $s_3\in(s_2,t)$ with 
\[
\SV_{s_2}^{s_3}(F)>\frac{M}{2}\quad\mbox{and}\quad \SV_{[s_3,t)}(F)>\frac{M}{2}.
\]
If we proceed in this way, we obtain an increasing sequence $\{s_n\}_n$ in $(a,t)$ such that
\[
\lim_{n\to\infty}s_n=t\quad\mbox{and}\quad\SV_{s_n}^{s_{n+1}}(F)>\frac{M}{2},\quad n\in\N.
\]
Having this in mind, for each $n\in\N$, we can find a division 
$D_n=\{\alpha_0^{(n)},\alpha_1^{(n)},\ldots,\alpha_{\nu_n}^{(n)}\}$ of $[s_n,s_{n+1}]$ and $x_j^{(n)}\in X$, $j=1,\dots,\nu_n$  
with $\|x_j^{(n)}\|_X\leq 1$ such that
\[
\Big\|\sum_{j=1}^{\nu_n}[F(\alpha_j^{(n)})-F(\alpha_{j-1}^{(n)})]x_j^{(n)}\Big\|_Y>\frac{M}{2}
\]

Let
\[
y_n=\sum_{j=1}^{\nu_n}[F(\alpha_j^{(n)})-F(\alpha_{j-1}^{(n)})]x_j^{(n)}\quad\mbox{for \ }n\in\N.
\]
We claim that $\sum_{n=1}^\infty|y^*(y_n)|<\infty$ for every $y^*\in Y^*$ with $\|y^*\|_{Y^*}\leq 1$. Indeed, given $N\in\N$, we have
\begin{align*}
\sum_{n=1}^N|y^*(y_n)|&=\sum_{n=1}^N\Big|\sum_{j=1}^{\nu_n}y^*\Big([F(\alpha_j^{(n)})-F(\alpha_{j-1}^{(n)})]x_j^{(n)}\Big)\Big|
	\\&
		\leq\sum_{n=1}^N\sum_{j=1}^{\nu_n}\|y^*\circ F(\alpha_j^{(n)})-y^*\circ F(\alpha_{j-1}^{(n)})\|_{X^*}
	\\&\leq	\sum_{n=1}^N\var_{s_n}^{s_{n+1}}(y^*\circ F)=\var_{s_1}^{s_{N+1}}(y^*\circ F).
\end{align*}
which together with Theorem \ref{SV-via-X*} leads to
\[
\sum_{n=1}^N|y^*(y_n)|\leq \SV_{s_1}^{s_{N+1}}(F)\leq \SV_{[s_1,t)}(F)<\infty.
\]
Thus, we conclude that the series $\sum_{n=1}^\infty y_n$ is weakly (unconditionally) convergent. 
Since $Y$ is weakly sequentially complete, it follows that  $\sum_{n=1}^\infty y_n$ converges in $Y$ 
(see Theorem \ref{wsc-series} in the Appendix). This contradicts   
the fact that $\|y_n\|_Y>\frac{M}{2}>0$ for every $n\in\N$.

In summary, we conclude that \eqref{limit-1} holds for every function from $SV(\ab,L(X,Y))$. Analogously we can show that \eqref{limit-2} is also true. 
\end{proof}

\smallskip

The results above, together with \cite[Corollary I.3.2]{H}, lead to the following conclusion about the continuity of a function of bounded semivariation.

\smallskip

\begin{corollary}\label{cont-a.e.}
Let $F\in SV(\ab,L(X,Y))$. If $Y$ is a weakly sequentially complete Banach space, then $F$ is continuous on $\ab$ except for a countable set.
\end{corollary}

\smallskip

\begin{remark}
Recalling that reflexive spaces are weakly sequentially complete (see \cite[Theorem 2.10.3]{HP}), 
Lemma \ref{WSC}, as well as Corollary \ref{cont-a.e.},  remains valid for $Y$ reflexive. 
\end{remark}

\smallskip

According to the characterization given in Theorem \ref{SV-via-X*}, for $F\in SV(\ab,L(X,Y))$ and $y^*\in Y^*$, 
the function $y^*\circ F:\ab\to X^*$ is of bounded variation on $\ab$. This implies that for each $t\in\ab$ both limits
\[
\lim_{\delta\to 0+}y^*\circ F(t-\delta)\mbox{ \ and \ }\lim_{\delta\to 0+}y^*\circ F(t+\delta)
\]
exist in $X^*$. Such limits can described by means of an operator mapping $X$ into the second dual $Y^{**}$ of $Y$.  
Now, we need to fix some notation to make our statement more precise.

\medskip

Given $U\in L(X,Y^{**})$ and $y^*\in Y^*$, we can define a linear functional $y^*\bullet U:X\to\R$ by setting 
$(y^*\bullet U)(x)=\big(U(x)\big)(y^*)$ for every $x\in X$.

\smallskip

\begin{theorem}
Let $F\in SV(\ab,L(X,Y)$. Then, for each $t\in(a,b]$ and $s\in[a,b)$, there exist $F(t-),\,F(s+)\in L(X,Y^{**})$ such that, for every $y^*\in Y^*$,
\[
\lim_{\delta\to 0+}y^*\circ F(t-\delta)=y^*\bullet F(t-)
\quad\mbox{and}\quad
\lim_{\delta\to 0+}y^*\circ F(s+\delta)=y^*\bullet F(s+)
\]
where $y^*\circ F$ is as in \eqref{dual-circ}.
\end{theorem}
\begin{proof}
Without loss of generality, let us assume $F(a)=0$. Given $t\in(a,b]$, for each $y^*\in Y^*$ there exists $T_{y^*}\in X^*$ such that
\[
\lim_{\delta\to 0+}y^*\circ F(t-\delta)=T_{y^*}.
\]
Considering $T:Y^*\to X^*$ defined by $T(y^*)=T_{y^*}$, $y^*\in Y^*$, clearly $T$ is linear. Moreover, by Theorem \ref{SV-via-X*}, 
\[
\|y^*\circ F(t-\delta)\|_{X^*}\leq \|y^*\|_{Y^*}\SV_a^b(F)\mbox{ \ for every \ }y^*\in Y^*,
\]
hence $T\in L(Y^*,X^*)$ with $\|T\|_{L(Y^*,X^*)}\leq \SV_a^b(F)$.

Let $T^\times:X\to Y^{**}$ be the mapping which associates to each $x\in X$ the linear functional $x^\times:Y^*\to\R$ given by 
$x^\times(y^*)=T_{y^*}x$ for $y^*\in Y^*$. Note that, for every $x\in X$ and $y^*\in Y^*$, we have
\[
\lim_{\delta\to 0+}(y^*\circ F(t-\delta))x=T_{y^*}x=x^\times(y^*)=(y^*\bullet T^\times)(x).
\]
Therefore $F(t-)=T^\times\in L(X,Y^{**})$ is the desired operator. Similarly, we can construct $F(s+)\in L(X,Y^{**})$ for $s\in[a,b)$.
\end{proof}

\smallskip

The theorem above suggests that functions of bounded semivariation are regulated in some weak sense. 
For operator-valued functions a more general notion of regulated function can be defined.

\smallskip

\begin{definition}\label{s-reg}
Given $F:\ab\to L(X,Y)$, we say $F$ is simply regulated on $\ab$ if, for each $x\in X$, the function 
$t\in\ab \longmapsto F(t)\,x\in Y$ is regulated. We will denote the set of such functions by $SG(\ab, L(X,Y))$.
\end{definition}

\smallskip

From the Banach-Steinhaus Theorem (c.f. \cite[Theorem 2.11.4]{HP}), given a function $F\in SG(\ab, L(X,Y))$, for each $t\in(a,b]$   
there exists $F(t\dot{-})\in L(X,Y)$ such that 
\begin{equation*}
\lim_{s\to t-} F(s)x=F(t\dot{-})x\mbox{ \ for every \ }x\in X.
\end{equation*}
Analogously, for $t\in[a,b)$, we have $F(t\dot{+})\in L(X,Y)$ satisfying $\lim\limits_{s\to t+} F(s)x=F(t\dot{+})x$ for every $x\in X$.

\medskip

The concept of simply regulated function appears in the literature under different nomeclatures (see \cite{H} and \cite{Sch1}), 
for instance, weakly regulated or $(\mathcal B)$-regulated with respect to the bilinear triple $\mathcal B=(L(X,Y),X,Y)$. Our choice  
follows the work of Honig in \cite{H-f}, among other of his publications and followers (see also \cite{Ba}). 
In some sense, such terminology could be seen as reference to 
the notion of regulated function in the weak* topology - also known as simple topology. 

\medskip

By the Definition \ref{s-reg}, it is clear 
\[
G(\ab, L(X,Y))\subset SG(\ab, L(X,Y))
\] (for details, see \cite[Proposition~3]{Sch1})

\smallskip

Recalling that $BV(\ab, X)\subset G(\ab, X)$, we could expect that a similar relation would hold in the frame of functions of bounded semivarition 
relatively to the notion of simply regulated functions defined above. The following example, inspired by \cite{BGC}, shows that this is not the case.

\smallskip

\begin{example}
Let $\ell_\infty$ be the Banach space of bounded sequences $x=\{x_n\}_n$ in $\R$, endowed with the usual supremum norm
\[
    \|x\|_{\infty}=\sup\{\,|x_n|:n\in\N\}.
\]

Denote by $e_k$, $k\in\N$, the canonical basis of $\ell_\infty$, where $e_k$ is the sequence which is $1$ in the $k$-th coordinate and null elsewhere. 

Consider the function $F:[0,1]\to L(\ell_\infty)$ given by
\[
\big(F(t)\big)\,x=\left\{\begin{array}{cl}
\dis x_1\,e_n&\mbox{if \ }t\in(\frac{1}{n+1},\frac{1}{n}], ~ n\in\N,
\\[4mm]
0&\mbox{if \ }t=0
\end{array}
\right.
\]
for $t\in[0,1]$ and $x=\{x_n\}_n\in\ell_\infty$.

Note that, for every $k\in\N$,
\[
\|\textstyle[F(\frac{1}{k})-F(\frac{1}{k+1})]\,e_1\|_\infty=\|e_k-e_{k+1}\|_\infty=1.
\]
Hence $\lim_{k\to\infty}\textstyle\big(F(\frac{1}{k})\big)e_1$ does not exist and, consequently, neither do $F(0+)$. 
This shows that $F$ is not simply regulated.

\smallskip

Let us prove that $F\in SV([0,1], L(\ell_\infty))$. Given a division $D=\{\alpha_0,\alpha_1,\ldots,\alpha_{\nu(D)}\}$ 
of $[0,1]$, let $k_j=\max\{k\in\N\,;\,k\alpha_j\leq 1\}$ for $j=1,\ldots,\nu(D)$. Considering $x_j\in\ell_\infty$, 
$x_j=\{x_n^{(j)}\}_n$, $j=1,\ldots,\nu(D)$ with $\|x_j\|_\infty\leq 1$, we have
\[
\sum_{j=1}^{\nu(D)}[F(\alpha_j)-F(\alpha_{j-1})]x_j
	=x^{(1)}_1\,e_{k_1}+\sum_{j=2}^{\nu(D)}\big[x^{(j)}_1\,e_{k_j}-x^{(j)}_1\,e_{k_{j-1}}\big]
\]
Taking $\Lambda=\{j\,:\,k_j\not= k_{j-1}\}\subset\{2,\ldots,\nu(D)-1\}$, we can write
\begin{align*}
\sum_{j=1}^{\nu(D)}[F(\alpha_j)-F(\alpha_{j-1})]x_j
	&=x^{(1)}_1\,e_{k_1}+\sum_{j\in\Lambda}\big[x^{(j)}_1\,e_{k_j}-x^{(j)}_1\,e_{k_{j-1}}\big]+x^{(\nu(D))}_1\,e_{k_{\nu(D)}}
	\\&=\sum_{j\in\Lambda\cup\{1\}}\lambda_j\,e_{k_j}+x^{(\nu(D))}_1\,e_{k_{\nu(D)}}
\end{align*}
where, for each $j\in\Lambda\cup\{1\}$, $\lambda_j$ corresponds to the difference between two elements of the set $\{x^{(i)}_1\,:\,i=1,\ldots,\nu(D)-1\}$. 
Clearly $|\lambda_j|\leq 2$, thus
\[
\Big\|\sum_{j=1}^{\nu(D)}[F(\alpha_j)-F(\alpha_{j-1})]x_j\Big\|_\infty\leq 2
\]
which implies that $\SV_0^1(F)<\infty$.
\end{example}

\smallskip

In view of this, a quite natural question arises: under which conditions is the space $SV(\ab,L(X))$ 
contained in the set of simply regulated functions? 

In \cite[Theorem 1]{Ba} it was proved that the inclusion holds whenever $X$ is a uniformly convex Banach space. 
Later, a final answer was given in \cite{BGC}, where a necessary and sufficient condition was established.
 
Aiming to present such result here, we have to consider a very special class of spaces, namely, 
all Banach spaces which do not contain an isomorphic copy of $c_0$ (by $c_0$ we denote the space of sequences in $\R$ 
converging to zero with respect to the supremum norm). By the Theorem of Bessaga and Pelczinsky (see Theorem \ref{Theo-BP} in the Appendix), 
the fact that a Banach space $X$ does not contain a copy of $c_0$ is equivalent to the following property:
\begin{itemize}
\item[{\rm(BP)}]{\it all series $\sum x_n$ in $X$ such that $\sum |x^*(x_n)|<\infty$ for every $x^*\in X^*$ are unconditionally convergent}.
\end{itemize}

Using this caractherization, we will present in details the relation between the sets 
$SV(\ab,L(X))$ and $SG(\ab,L(X))$ described in \cite[Theorem 5]{BGC}. 

\smallskip

\begin{theorem}\label{SV-SG}
The following assertions are equivalent:
\begin{enumerate}[$(1)$]
\item $X$ does not contain an isomorphic copy of $c_0$
\item every function $F:\ab\to L(X)$ of bounded semivariation is simply regulated.
\end{enumerate}
\end{theorem}

\smallskip

The proof of this theorem is contained in the following two lemmas.

\smallskip

\begin{lemma}
If $X$ does not contain an isomorphic copy of $c_0$, then \[SV(\ab,L(X))\subset SG(\ab,L(X)).\]
\end{lemma}
\begin{proof}
Given $F\in SV(\ab,L(X))$ and $x\in X$, $x\not=0$, let $F_x:\ab\to X$ be the function given by 
\[
F_x(t)=F(t)\,x\quad\mbox{for \ }t\in\ab.
\]
Fixed an arbitrary $t\in(a,b]$, to show that the left-sided limit $F_x(t-)$ exists, 
consider an increasing sequence $\{t_n\}_n$ on  $(a,t)$ converging to $t$. 

Let $x^*\in X^*$. For $N\in\N$, taking the division $D_N=\{t_0,t_1,\ldots,t_N, b\}$ of $\ab$ 
formed by elements of the sequence $\{t_n\}_n$ and $t_0=a$, we have
\[
\sum_{j=1}^N|x^*(F_x(t_j)-F_x(t_{j-1}))|=x^*\left(\sum_{j=1}^N[F(t_j)-F(t_{j-1})]\,\lambda_j\, x\right)
\]
where $\lambda_j={\rm sgn}(x^*(F_x(t_j)-F_x(t_{j-1})))$ for $j=1,\ldots, N$. If we put $x_j=\frac{\lambda_j \, x}{\|x\|_X}$, we get
\begin{align*}
\sum_{j=1}^N|x^*\big(F_x(t_j)-F_x(t_{j-1})\big)|&\leq\|x^*\|_{X^*}\|x\|_X\left\|\sum_{j=1}^N[F(t_j)-F(t_{j-1})]\,x_j\right\|_X
\\&\leq \|x^*\|_{X^*}\,\|x\|_X\,\SV_a^b(F).
\end{align*}
Since the inequality is valid for every $N\in\N$, we conclude that
\[ 
\sum_{n=1}^\infty|x^*(F_x(t_n)-F_x(t_{n-1}))|<\infty,\quad x^*\in X^*.
\]
By the property {\rm (BP)} of the space $X$, the series $\sum_{n=1}^\infty\big(F_x(t_n)-F_x(t_{n-1})\big)$ 
converges to some $z\in X$ and, consequently,
\[
\lim_{n\to\infty}F(t_n)\,x=\lim_{n\to\infty}\sum_{k=1}^n\big(F_x(t_k)-F_x(t_{k-1})\big)+F_x(a)=z+F(a)\,x.
\]

It remains to show that the limit does not depend on the choice of the sequence $\{t_n\}_n$. 
To this aim, let $\{s_n\}_n$ be another increasing sequence with $\lim_{n\to\infty}s_n=t$. By the same argument 
used above, there exists $\tilde{z}\in X$ such that 
\[
\tilde{z}=\sum_{n=1}^\infty\big(F_x(s_n)-F_x(s_{n-1})\big)
\quad\mbox{and}\quad
\lim_{n\to\infty}F_x(s_n)=\tilde{z}+F(a)\,x.
\]
Ordering the set $\{t_n:n\in\N\}\cup\{s_n:n\in\N\}$ we obtain an increasing sequence $\{r_n\}_n$ converging to $t$ whose series 
$\sum_{n=1}^\infty\big(F_x(r_n)-F_x(r_{n-1})\big)$ also converges. Moreover, the limit $\lim_{n\to\infty}F_x(r_n)$ exists.  
Since $\{F_x(t_n)\}_n$ and $\{F_x(s_n)\}_n$ are convergent subsequences of $\{F_x(r_n)\}_n$, we should have
\[
\lim_{n\to\infty}F_x(t_n)=\lim_{n\to\infty}F_x(r_n)=\lim_{n\to\infty}F_x(s_n)
\]
which proves that $\tilde{z}=z$ and $F_x(t-)=z+F(a)\,x$.
 
Similarly, we can show that the right-sided limit of $F_x$ exists for every $t\in [a,b)$.
\end{proof}

\smallskip

The second lemma gives the reverse implication from Theorem \ref{SV-SG}. Roughly speaking, 
we will show that if $c_0$ is isomorphically embedded into the space $Y$, 
one can construct a function $F:\ab\to L(X)$ of bounded semivariation which is not simply regulated.

\smallskip

\begin{lemma}
If $SV(\ab,L(X))\subset SG(\ab,L(X))$, then $X$ does not contain a copy of $c_0$.
\end{lemma}
\begin{proof}
By contradiction, assume that $X$ contains a isomorphic copy of $c_0$ and denote it by $Z$. Let 
$\psi: c_0\to Z$ be an isomorphism and put $z_k:=\psi(e_k)$ where
$e_k,$ $k\in\N$, stands for the canonical Schauder basis of $c_0$.

It is known that there exist positive constants $C_1$ and $C_2$ such that, for $N\in\N$, taking $\lambda_j\in\R$, $j=1,\ldots,N$, 
we have
\[
C_1\,\sup_{1\leq j\leq N}|\lambda_j|\leq\Big\|\sum_{j=1}^N\lambda_j\,e_j\Big\|_{\infty}
\leq C_2\,\sup_{1\leq j\leq N}|\lambda_j|,
\]
(see [Kadets, Theorem 6.3.1], [Diestel, Theorem V.6]). Thus, 
by the fact that $Z$ and $c_0$ are isomorphic,
\begin{equation}\label{seq}
C_1\,\sup_{1\leq j\leq N}|\lambda_j|\leq\Big\|\sum_{j=1}^N\lambda_j\,z_j\Big\|_X
\leq C_2\,\sup_{1\leq j\leq N}|\lambda_j|,
\end{equation}
for $N\in\N$ and $\lambda_j\in\R$, $j=1,\ldots,N$.

Using the sequence $z_n$, $n\in\N$, mentioned above and its properties we will construct a function $F:\ab\to L(X)$ in a few steps.

\medskip

\noindent {\it Step 1.} Clearly, $z_n$, $n\in\N$, defines a basis for $Z$ and, for each $k\in\N$, the projection $\pi_k:Z\to\R$, given by 
$\pi_k(\sum_n\lambda_n\,z_n)=\lambda_k$, is continuous (see [Diestel, p. 32]). Since \eqref{seq} implies that
\[
\big|\pi_k\Big(\sum_{n=1}^N z_n\Big)\big|\leq\frac{1}{C_1}\Big\|\sum_{n=1}^N z_n\Big\|_Z,\quad N\in\N,
\]
we have $\|\pi_k\|_{Z^*}\leq \frac{1}{C_1}$ for every $k\in\N$.

\medskip

\noindent {\it Step 2.} For $k\in\N$, let $S_k:Z\to Z$ be given by 
\[
S_k(x)=\sum_{n=1}^k\pi_n(x)\,z_{2^k+n},\quad\mbox{for \ } x\in Z
\]
Note that, $S_k$ is a bounded linear operator on $Z$ for every $k\in\N$. Indeed, given $x\in Z$, we can write
\[
S_k(x)=\sum_{n=1}^k\pi_n(x)\,z_{2^k+n}=\sum_{j=1}^{2^k+k}\beta_j z_j
\]
where $\beta_j=\pi_n(x)$ if $j=2^k+n$ for some $n=1,\ldots,k$, otherwise $\beta_j=0$. Thus, by \eqref{seq},
\[
\|S_k(x)\|_Z\leq C_2 \sup_{1\leq j\leq 2^k+k}|\beta_j|= C_2 \sup_{1\leq n\leq k}|\pi_n(x)|\leq C_2 \sup_{1\leq n\leq k}\|\pi_n\|_{Z^*}\|x\|_Z,
\]
which implies that $\|S_k\|_{L(Z)}\leq \frac{C_2}{C_1}$ for every $k\in\N$.

\medskip

\noindent {\it Step 3.} Given $j,\,k\in\N$, put $f_{k,j}=\pi_j\circ S_k$. By the Hahn-Banach theorem, the functional  $f_{k,j}\in Z^*$  
can be extended to a continuous linear functional $\tilde{f}_{k,j}$ on $X$ satisfying
\begin{equation}\label{func}
\|\tilde{f}_{k,j}\|_{X^*}=\|f_{k,j}\|_{Z^*}\leq \|\pi_j\|_{Z^*}\|S_k\|_{L(Z)}\leq \frac{C_2}{(C_1)^2}.
\end{equation}

\medskip

\noindent {\it Step 4.} For $x\in X$ and $k\in\N$, let $T_k(x)=\sum_{j=1}^k\tilde{f}_{k,2^k+j}(x)\,z_{2^k+j}$.
Clearly, $T_k\in L(X)$ and it follows from \eqref{seq} and \eqref{func} that 
\[
\|T_k(x)\|_X=\Big\|\sum_{j=1}^{k}\tilde{f}_{k,2^k+j}(x)\,z_{2^k+n}\Big\|_X
\leq C_2 \sup_{1\leq j\leq k}|\tilde{f}_{k,2^k+j}(x)|
\leq\Big(\frac{C_2}{C_1}\Big)^2\|x\|_X\,,\quad x\in X,
\]
that is, $\|T_k\|_{L(X)}\leq \big(\frac{C_2}{C_1}\big)^2$ for all $k\in\N$.

\medskip

We are now ready to define $F:\ab\to L(X)$. Given $x\in X$, let 
\[
F(t)x=T_k(x)\quad\mbox{for \ }t\in(t_{k+1},t_k]
\]
where $t_k=a+\frac{(b-a)}{k}$, $k\in\N$. 

It is not hard to see that $F$ is not simply regulated. Indeed, for each $k\in\N$, noting that $S_k(z_1)=z_{2^k+1}$, we get
\[
F(t_k)z_1=T_k(z_1)=\sum_{j=1}^kf_{k,2^k+j}(z_1)\,z_{2^k+j}= \sum_{j=1}^k\pi_{2^k+j}\big(S_k(z_1)\big)z_{2^k+j}=z_{2^k+1},
\]
which by \eqref{seq} leads to
\[
\|[F(t_k)-F(t_{k+1})]\,z_1\|_X=\|z_{2^k+1}-z_{2^{k+1}+1}\|_X\geq C_1.
\]
Hence the limit $\lim_{t\to a+} F(t)\,z_1$ does not exist.

Now, we will show that $F\in SV(\ab,L(X))$. Considering a division $D\in\mathcal D\ab$, with 
$D=\{\alpha_0,\alpha_1,\ldots,\alpha_{\nu(D)}\}$, let $k_j\in\N$ be such that $\alpha_j\in(t_{k_j+1},t_{k_j}]$, $j=1,\ldots,\nu(D)$.
For $x_j\in X$, $j=1,\ldots,\nu(D)$ with $\|x_j\|_X\leq 1$ we have
\[
\sum_{j=1}^{\nu(D)}[F(\alpha_j)-F(\alpha_{j-1})]x_j
=T_{k_1}(x_1)+\sum_{j=2}^{\nu(D)}[T_{k_j}(x_j)-T_{k_{j-1}}(x_j)]
\]
Taking $\Lambda=\{j\,:\,k_j\not= k_{j-1}\}\subset\{2,\ldots,\nu(D)-1\}$, we can write
\[
\sum_{j=1}^{\nu(D)}[F(\alpha_j)-F(\alpha_{j-1})]x_j
	=\sum_{j\in\Lambda\cup\{1\}}T_{k_j}(y_j)+T_{k_{\nu(D)}}(x_{\nu(D)})
\]
where, for each $j\in\Lambda\cup\{1\}$, $y_j\in X$ corresponds to the difference between two elements of the set $\{x_i\,:\,i=1,\ldots,\nu(D)-1\}$. 
Noting that $\|y_j\|_X\leq 2$, by \eqref{seq} and \eqref{func}, it follows that
\begin{align*}
\Big\|\hskip-1.2mm\sum_{j\in\Lambda\cup\{1\}}T_{k_j}(y_j)\Big\|_X
	&=\Big\|\hskip-1.2mm\sum_{j{\in}\Lambda\cup\{1\}}\sum_{n=1}^{k_j}\tilde{f}_{k_j,2^{k_j}+n}(y_j)\,z_{2^{k_j}+n}\Big\|_X
	\\&\leq C_2 \sup_{\stackrel{1\leq n\leq k_j}{j\in\Lambda\cup\{1\}}}\hskip-1.3mm|\tilde{f}_{k_j,2^{k_j}+n}(y_j)|
	\leq 2\Big(\frac{C_2}{C_1}\Big)^2
\end{align*}
Therefore,
\[
\Big\|\sum_{j=1}^{\nu(D)}[F(\alpha_j)-F(\alpha_{j-1})]\,x_j\Big\|_X
\leq 2\Big(\frac{C_2}{C_1}\Big)^2+\|T_{k_{\nu(D)}}(x_{\nu(D)})\|_X\leq 3\Big(\frac{C_2}{C_1}\Big)^2
\]
wherefrom it follows that $\SV_a^b(F)<\infty$.

In summary, $F\in SV(\ab,L(X))$ and $F$ is not simply regulated, which is a 
contradiction. Thus the lemma is established.
\end{proof}

\section{Semivariation and the Kurzweil integral}
In the recent years non-absolute integrals have been increasingly investigated. Among them it is worth highlighting the 
one due to Kurzweil, \cite{Ku1}, whose concept of integration has been the background of several 
papers related to differential and difference equations. See, for instance, \cite{sla}, \cite{FMS1} and \cite{MS}.

This section is dedicated to investigate the connection between the semivariation and the integral due to Kurzweil  
in two different aspects. First, we present a result by Honig which generalizes the following fact: {\it every function of bounded variation 
is a multiplier for Kurzweil integrable functions}. Next, we apply the concept of semivariation to derive two convergence results
for Stieltjes type integral and we conclude the section by proving a new characterization of 
semivariation by the means of the abstract Kurzweil-Stieltjes integral.

In what follows we deal with special cases of the 
integral introduced by J.~Kurzweil in \cite{Ku1} under the name ``generalized Perron integral". 
For the reader's convenience, let us recall its definition.

\medskip 

As usual, a {\it partition of }$\ab$ is a tagged division $P=(\tau_j,[\alpha_{j-1},\alpha_j])$ where 
the set $\{\alpha_0,\alpha_1,\ldots,\alpha_{\nu(P)}\}$ is a division of $\ab$ and $\tau_j\in[\alpha_{j-1},\alpha_j]$ for $j=1,\ldots,\nu(P)$. 
A {\it gauge on }$\ab$ is a positive function $\delta:\ab\to\R^+$. Furthermore, given a gauge $\delta$ on $\ab$, a partition 
$P=(\tau_j,[\alpha_{j-1},\alpha_j])$ is called $\delta$-fine if
\[
[\alpha_{j-1},\alpha_j]\subset(\tau_j-\delta(\tau_j),\tau_j+\delta(\tau_j))\quad\mbox{for \ } j=1,\ldots,\nu(P).
\]
Given an arbitrary gauge $\delta$ on $\ab$, the existence of (at least one) $\delta$-fine partition is a known result, 
the so-called Cousin's lemma (see \cite[Theorem 4.1]{Henstock} or \cite[Lemma 1.4]{Sch}).

\medskip

A function $U:[a,b]\times [a,b]\to X$ is 
Kurzweil integrable on $[a,b]$, if there exists $I\in X$ such that 
for every $\varepsilon>0$, there is a gauge $\delta$ on $\ab$ such that
\[
\left\|\sum_{j=1}^{\nu(P)}[U(\tau_j,\alpha_{j})-U(\tau_j,\alpha_{j{-}1})]- I\right\|_X<\varepsilon
\mbox{ \ for all $\delta$-fine partitions of $\ab$.}
\]
In this case, we define the Kurzweil integral as $\int_a^b DU(\tau,t)=I$.

For a more comprehensive study of the properties of the Kurzweil integral we refer to the monograph \cite{Sch} and references therein.

Taking $U(\tau,t)=f(\tau)\,t$ for $t\in\ab$, where $f:[a,b]\to X$ is a given function,
the definition above corresponds to an integration process based on Riemann-type sums, namely, the Henstock-Kurzweil integral.
Such integral is known to extend the theory of Lebesgue integral. 
In what follows, when dealing with the Kurzweil-Henstock integral we will write simply $\int_a^b f(t)\dd t$ instead of $\int_a^b D[f(\tau)t]$.

Secondly, we are interested in the abstract Kurzweil-Stieltjes integrals $\int_a^bF\,\dd[g]$ and $\int_a^b\dd[F]\,g$,  
where $F:\ab\to L(X)$ and $g:[a,b]\to X$ (see \cite{Sch1}). These integrals are obtained from the choices  
$U(\tau,t)=F(\tau)\,g(t)$ and $U(\tau,t)=F(t)\,g(\tau)$ for $t,\tau\in\ab$, respectively. 

\medskip

In the sequel we state two existence results for the abstract Kurzweil-Stieltjes integral relatively to functions of bounded semivariation  
(for the proof see \cite[Thereom 3.3]{MT}).

\smallskip

\begin{theorem}\label{existence}
Let $F\in SV(\ab,L(X,Y))$.
\begin{enumerate}[$(i)$]
\item If $g\in G(\ab,X)$, then the integral $\int_a^bF\,\dd[g]$ exists and
\[
\Big\|\int_a^bF\,\dd[g]\Big\|_X\leq \big(\|F(a)\|_X+\|F(b)\|_X+\SV_a^b(F)\big)\,\|g\|_\infty.
\]
\item If $F\in SG(\ab,L(X))$ and $g\in G(\ab,X)$, then the integral $\int_a^b\dd[F]g$ exists and
\[
\Big\|\int_a^b\dd[F]\,g\Big\|_X\leq \SV_a^b(F)\,\|g\|_\infty.
\]
\end{enumerate}
\end{theorem}

\smallskip

Let us denote by $\mathcal K(\ab,X)$ the set of all Henstock-Kurzweil integrable functions, that is, 
all functions $f:\ab\to X$ whose integral $\int_a^b f(t)\dd t$ exists. The linearity of the integral implies that $\mathcal K(\ab,X)$ is a linear space 
(cf. \cite{G} or \cite{Sch}). 

The following theorem, borrowed from \cite[1.15]{H2}, shows that the functions of 
bounded semivariation are multipliers for the space $\mathcal K(\ab,X)$.

\smallskip

\begin{theorem}\label{K-multi}
Let $g\in\mathcal K(\ab,X)$ and $F\in SV(\ab,L(X))$. 
Consider the function $Fg:\ab\to X$ given by $(Fg)(t)=F(t)g(t)$ for $t\in\ab$. Then $Fg\in\mathcal K(\ab,X)$ and
\begin{equation}\label{by-parts}
\int_a^bF(t)g(t)\,\dd t=\int_a^bF\,\dd [\tilde{g}],
\end{equation}
where $\tilde{g}(t)=\int_a^tg(s)\,\dd s$ for $t\in\ab$.
\end{theorem}
\begin{proof}
First of all, noting that the indefinite integral of $g$ defines a continuous function on $\ab$ (cf. \cite[Theorem 9.12]{G}), 
it is clear by Theorem \ref{existence} that $\int_a^bF\,\dd \tilde{g}$ exists.

Given $\varepsilon>0$, let $\delta_1$ and $\delta_2$ be 
gauges on $\ab$ such that
\begin{equation}\label{eq-1}
  \Bigg\|\sum_{j=1}^{\nu(P)}F(\tau_j)[\tilde{g}(\alpha_{j})-\tilde{g}(\alpha_{j{-}1})]-\int_a^b F\,\dd[\tilde{g}]\Bigg\|_X<\varepsilon
  \mbox{ \ for all $\delta_1$-fine partitions of $\ab$,}
\end{equation}
and
\[
  \Bigg\|\sum_{j=1}^{\nu(P)}g(\tau_j)(\alpha_{j}-\alpha_{j{-}1})- \int_a^bg(s)\,\dd s\Bigg\|_X<\varepsilon
\mbox{ \ for all $\delta_2$-fine partitions of $\ab$.}
\]

Due to the Saks-Henstock Lemma (see \cite[Lemma 1.13]{Sch}), for any $\delta_2$-fine partition of $\ab$, 
$P\,{=}\,(\tau_j,[\alpha_{j-1},\alpha_j])$, we have
\begin{equation}\label{eq-2}
  \Bigg\|\sum_{k{=}j}^{\nu(P)} \Big[g(\tau_k)(\alpha_{k}-\alpha_{k{-}1})
               \,{-}\,\int_{\alpha_{k{-}1}}^{\alpha_k}g(s)\,\dd s\Big]\Bigg\|_X<\varepsilon
  \quad\mbox{for \ } j=1,2,\dots,\nu(P).
\end{equation}

Put $\delta(t)=\min\{\delta_1(t),\,\delta_2(t)\}$ for $t \in \ab$. 
Given a $\delta$-fine partition $P=(\tau_j,[\alpha_{j-1},\alpha_j])$ of $\ab$, by \eqref{eq-1} we get
\begin{align*}
   &\Bigg\|\sum_{j=1}^{\nu(P)}F(\tau_j)g(\tau_j)(\alpha_{j}-\alpha_{j{-}1})-\int_a^b F\,\dd[\tilde{g}]\Bigg\|_X
  \\&\quad
   \le\Big\|\sum_{j=1}^{\nu(P)}F(\tau_j)g(\tau_j)(\alpha_{j}-\alpha_{j{-}1})-\sum_{j=1}^{\nu(P)}F(\tau_j)[\tilde{g}(\alpha_{j})-\tilde{g}(\alpha_{j{-}1})]\Big\|_X
  \\&\qquad
   +\Bigg\|\sum_{j=1}^{\nu(P)}F(\tau_j)[\tilde{g}(\alpha_{j})-\tilde{g}(\alpha_{j{-}1})]-\int_a^b F\,\dd[\tilde{g}]\Bigg\|_X
  \\&\quad < 
   \Bigg\|\sum_{j{=}1}^{\nu(P)} F(\tau_j)\Big[g(\tau_j)(\alpha_{j}-\alpha_{j{-}1})-\int_{\alpha_{j{-}1}}^{\alpha_j}g(s)\dd s\Big]\Bigg\|_X+\varepsilon
\end{align*}

In order to estimate the other term in the last inequality, we will make use of the following equality mentioned in \cite{H2}:
\[
   \sum_{j=1}^m A_j\,x_j
  =\sum_{j=1}^m[A_j\,{-}\,A_{j{-}1}]\left(\sum_{k{=}j}^m x_k\right)\,{+}\,A_0\left(\sum_{k{=}1}^m x_k\right)
\]
for all $A_j\in L(X)$ and all $x_j\in X$. Let us consider $m=\nu(P)$ and also
\[
A_0=F(a),\quad A_j=F(\tau_j),\quad x_j=g(\tau_j)(\alpha_{j}-\alpha_{j{-}1})-\int_{\alpha_{j{-}1}}^{\alpha_j}g(s)\dd s
\]
for $j=1,\ldots\nu(P)$. Note that, by \eqref{eq-2} we have $\Big\|\sum_{k{=}j}^{\nu(P)}x_k \Big\|_X\leq\varepsilon$ for each $j=1,\ldots\nu(P)$. 
Therefore
\begin{align*}
   &\Bigg\|\sum_{j{=}1}^{\nu(P)} F(\tau_j)\Big[g(\tau_j)(\alpha_{j}-\alpha_{j{-}1})-\int_{\alpha_{j{-}1}}^{\alpha_j}g(s)\dd s\Big]\Bigg\|_X
  \\&\qquad 
   <\varepsilon\Bigg\|\sum_{j{=}1}^{\nu(P)} [F(\tau_j)-F(\tau_{j{-}1})]\,
              \dis\frac{\sum_{k{=}j}^{\nu(P)}x_k }{\varepsilon}\Bigg\|_X
	+\Big\|F(a)\Big(\sum_{j{=}1}^{\nu(P)}x_j \Big)\Big\|_X
   \\&\qquad<\varepsilon\,\big(\SV_a^b(F)+\|F(a)\|_{L(X)}\big)
\end{align*}
(where $\tau_0=a$). Having all these in mind, we obtain
\[
\Bigg\|\sum_{j=1}^{\nu(P)}F(\tau_j)g(\tau_j)(\alpha_{j}-\alpha_{j{-}1})-\int_a^b F\,\dd[\tilde{g}]\Bigg\|_X
	<\varepsilon\,\big(1+\SV_a^b(F)+\|F(a)\|_{L(X)}\big)
\]
for all $\delta$-fine partitions of $\ab$, wherefrom we conclude that $\int_a^b F(t)\,g(t)$ exists and the unicity 
of the integral leads to \eqref{by-parts}.
\end{proof}

\smallskip

\begin{remark}
The theorem above is presented in \cite{H2} when 
integration by parts formulas for Henstock-Kurzweil integral are discussed. 
Indeed, taking into account the results from \cite{Sch4} (see also \cite[Corollary 3.6]{MT}), the equality \eqref{by-parts} can be rewritten as 
\[
\int_a^bF(t)g(t)\,\dd t=F(b)\tilde{g}(b)- \int_a^bF\,\dd \tilde{g}
\]
Moreover, due to the continuity of the function $\tilde{g}$, 
the Stieltjes-type integral in the formula above (as well as in \eqref{by-parts}) 
can be read as a Riemann-Stieltjes integral defined in the Banach space setting (see \cite[1.13]{H2}).

We would like also to remark that the result in Theorem \ref{K-multi} remains valid if we replace the function 
$g:\ab\to X$ by Henstock-Kurzweil integrable functions defined in $\ab$ and taking values in $L(X)$. 
\end{remark}

\smallskip

Now we turn our attention to the 
connection between semivariation and the abstract Kurzweil-Stieltjes integral. First, we focus in Helly type result, that is, 
convergence results for the integral based on assumptions similar to those presented in Lemma \ref{Helly}.  
The theorem in the sequel, to our knowlegde, is not available in literature in the presented formulation.

\smallskip

\begin{theorem}
Let $F:\ab\to L(X)$, a sequence $\{F_n\}_n\subset SV(\ab,L(X))$ and a 
constant $M>0$ be such that
\[
\SV_a^b(F_n)\leq M\quad\mbox{for every \ }n\in\N,
\]
and
\[
\lim_{n\to\infty}\|F_n(t)-F(t)\|_{L(X)}=0\quad\mbox{for every \ }t\in \ab.
\]
If $g\in G(\ab,X)$, then the integrals $\int_a^b F\,\dd[g]$ and $\int_a^b F_n\,\dd[g]$, $n\in\N,$ exist and
\begin{equation}\label{conv1}
   \lim_{n\to\infty}
   \int_a^b F_n\,\dd[g]=\int_a^b F\,\dd[g].
\end{equation}
\end{theorem}
\begin{proof}
By Lemma \ref{Helly} we know that $F\in SV(\ab,L(X))$, thus the existence of the integrals is guaranteed by Theorem \ref{existence} $(i)$. 

To prove the convergence, we first consider the case when $g$ is a finite step function. 
Due to the linearity of the integral, it is enough to show that \eqref{conv1} holds for functions of
the form $\chi_{[a,\tau]}x$, $\chi_{[\tau, b]}x$, $\chi_{[a]}x$ and $\chi_{[b]}x$, where $\tau\in(a,b)$ and $x\in X$.

Given $\tau\in [a,b)$ and $x\in X$, by \cite[Proposition 2.3.3]{T1} (with an obvious extension to Banach spaces-valued functions) we have
\[
  \int_a^b (F_n-F)\,\dd[\chi_{[a,\tau]}x]=F(\tau)x-F_n(\tau)x,
\]
hence \eqref{conv1} follows. 
Similarly, one can prove the equality for $\chi_{[\tau, b]}x$, $\chi_{[a]}x$ and $\chi_{[b]}x$.

Now, assuming $g\in G(\ab,X)$ and given $\varepsilon>0$, there exists a finite step function $\varphi:\ab\to X$ such that
$\|g-\varphi\|_\infty<\varepsilon$ (see \cite[Theorem I.3.1]{H}). Let $n_0\in\N$ be such that
\[
\|(F_n-F)(a)\|+\|(F_n-F)(b)\|<M \quad\mbox{and}\quad \Big\|\int_a^b (F_n-F)\,d[\varphi]\Big\|<\varepsilon
\]
for $n>n_0$. These inequalities, together with \eqref{linear} and Theorem \ref{existence}, imply that 
\begin{align*}
&\Big\|\int_a^b(F_n-F)\,\dd[g]\Big\|_X
\\&\qquad
\leq \Big\|\int_a^b(F_n-F)\,\dd[g-\varphi]\Big\|_X+\Big\|\int_a^b(F_n-F)\,\dd[\varphi]\Big\|_X
\\&\qquad
\leq\big(\|(F_n-F)(a)\|_X+\|(F_n-F)(b)\|_X+\SV_a^b(F_n-F)\big)\,\|g-\varphi\|_\infty+\varepsilon
\\&\qquad<\big(M+\SV_a^b(F_n)+\SV_a^b(F)\big)\,\varepsilon+\varepsilon<\varepsilon(3\,M+1)
\end{align*}
for every $n>n_0$, which proves \eqref{conv1}.
\end{proof}

\smallskip

We remark that in \cite{Sch-1992} the convergence result above is proved for real-valued functions of bounded variation. 

Still a Helly type result, the following theorem concerns integrals of the form $\int_a^b \dd[F]\,g$.

\smallskip

\begin{theorem}
Let $F:\ab\to L(X)$, $\{F_n\}_n\subset SV(\ab,L(X))\cap SG(\ab,L(X))$ and a 
constant $M>0$ be such that
\[
\SV_a^b(F_n)\leq M\quad\mbox{for every \ }n\in\N,
\]
and
\begin{equation}\label{lim-conv}
\lim_{n\to\infty}\Big(\sup_{t\in\ab}\|F_n(t)x-F(t)x\|_{X}\Big)=0\quad\mbox{for every \ }x\in X.
\end{equation}
If $g\in G(\ab,X)$, then the integrals $\int_a^b \dd[F]\,g$ and $\int_a^b \dd[F_n]\,g$, $n\in\N,$ exist and
\begin{equation}\label{conv2}
   \lim_{n\to\infty}
   \int_a^b \dd[F_n]\,g=\int_a^b \dd[F]\,g.
\end{equation}
\end{theorem}
\begin{proof}
Given $x\in X$, for each $n\in\N$ put $(F_n)_x:t\in\ab\longmapsto F_n(t)x\in X$. By \eqref{lim-conv} 
it follows that the sequence of regulated functions 
$\{(F_n)_x\}_n$ converges uniformly in $\ab$ to $F_x:t\in\ab\longmapsto F(t)x\in X$. Hence, by \cite[I.3.5]{H} the function $F_x$ is regulated 
and, since it holds for each $x\in X$, we have $F\in SG(\ab,L(X))$. Noting that $F\in SV(\ab,L(X))$ (see Lemma \ref{Helly}), 
the existence of the integral $\int_a^b \dd[F]\,g$, as well as the existence of $\int_a^b \dd[F_n]\,g$, $n\in\N,$ yields from 
Theorem \ref{existence} $(ii)$.

In order to prove \eqref{conv2}, we first consider the case when $g(t)=\chi_{[a,\tau]}(t)\tilde{x}$ for $t\in\ab$, where 
$\tau\in (a,b)$ and $\tilde{x}\in X$ are arbitrarily fixed. For each $n\in\N$ we have
\[
  \int_a^b \dd[F_n-F]g=\lim_{s\to \tau-} F_n(s)\tilde{x}-\lim_{s\to \tau-}F(s)\tilde{x}-[F_n(a)-F(a)]\tilde{x},
\]
(cf. \cite[Proposition 14]{Sch1}) or equivalently,
\begin{equation}\label{s1}
  \int_a^b \dd[F_n-F]g= F_n(\tau\dot{-})\tilde{x}- F(\tau\dot{-})\tilde{x}
  -[F_n(a)-F(a)]\tilde{x},
\end{equation}
where $F_n(\tau\dot{-}),\,F(\tau\dot{-})\in L(X)$ are operators satisfying
\begin{equation*}\label{simply}
\lim_{s\to \tau-} F_n(s)x=F_n(\tau\dot{-})x\quad\mbox{and}\quad\lim_{s\to \tau-} F(s)x=F(\tau\dot{-})x
\end{equation*}
for every $x\in X$. 
Given $\varepsilon>0$, by \eqref{lim-conv} there exists $n_0\in\N$ such that
\begin{equation}\label{s2}
\|[F_n(t)-F(t)]\tilde{x}\|_X<\frac{\varepsilon}{3}\quad \mbox{for \ }n\geq n_0 \mbox{ \ and \ }t\in\ab.
\end{equation}
Fixed $n\geq n_0$, we can choose $\delta>0$ such that 
\begin{equation}\label{s3}
\|F_n(\tau\dot{-})\tilde{x}-F_n(s)\tilde{x}\|_X\leq \frac{\varepsilon}{3}
\quad\mbox{and}\quad
\|F_n(\tau\dot{-})\tilde{x}-F_n(s)\tilde{x}\|_X\leq \frac{\varepsilon}{3}.
\end{equation}
Let us choose $s\in(\tau-\delta,\tau)$. From \eqref{s2} and \eqref{s3} it follows that
\begin{align*}
&\|F_n(\tau\dot{-})\tilde{x}- F(\tau\dot{-})\tilde{x}\|_X
\\&\qquad\leq
\|F_n(\tau\dot{-})\tilde{x}-F_n(s)\tilde{x}\|_X
+\|F_n(s)\tilde{x}-F(s)\tilde{x}\|_X+\|F(s)\tilde{x}-F(\tau\dot{-})\tilde{x}\|_X<\varepsilon,
\end{align*}
which together with \eqref{s2} applied to $t=a$, shows that the integral in \eqref{s1} tends to zero. 
With similar argument we can prove that \eqref{conv2} holds when $g$ is a function of the form 
$\chi_{[\tau, b]}x$, $\chi_{[a]}x$ and $\chi_{[b]}x$ for $\tau\in (a,b)$ and $x\in X$. As a consequence of the linearity of 
the integral we conclude that \eqref{conv2} is valid if $g$ is a step function.

Now, assuming that $g\in G(\ab,X)$ and given $\varepsilon>0$, let $\varphi:\ab\to X$ be a finite step function such that
$\|g-\varphi\|_\infty<\varepsilon$ (see \cite[Theorem I.3.1]{H}). Thus, by Theorem \ref{existence} $(i)$ we have 
\begin{align*}
\Big\|\int_a^b\dd[F_n-F]\,g\Big\|_X
&\leq \Big\|\int_a^b\dd[F_n-F]\,(g-\varphi)\Big\|_X+\Big\|\int_a^b\dd[F_n-F]\,\varphi\Big\|_X
\\& \leq\SV_a^b(F_n-F)\|g-\varphi\|_\infty+\Big\|\int_a^b\dd[F_n-F]\,\varphi\Big\|_X
\\& \leq 2\,M\varepsilon+\Big\|\int_a^b\dd[F_n-F]\,\varphi\Big\|_X.
\end{align*}
Since $\varphi$ is a step function, the result now follows from first part of the proof.
\end{proof}

\smallskip

Similar convergence results have been proved in \cite{MT-2} and \cite{MS} in the frame of functions of bounded variation.

\medskip

In literature the notion of variation is sometimes described by the means of different integrals of the Stieltjes type. 
In \cite{BK}, using the Young integral on Hilbert spaces, not only a characterization for the norm $\|\cdot\|_{BV}$ is presented 
but also the notion of essential variation is treated. 
Dealing with the semivariation and the interior integral (i.e. the Dushnik integral), it is worth highlighting \cite[Corollary I.5.2]{H}.

Inspired by those results, we present here a characterization of the semivariation via the abstract Kurzweil-Stieltjes integral. 
To this end, we will need the following estimates whose proofs are quite similar to \cite[Lemma 3.1]{MT}.

\smallskip

\begin{lemma}
Let $F:\ab\to L(X)$ and $g:\ab\to X$ be given. For every partition $P=(\tau_j,[\alpha_{j{-}1},\alpha_j])$ of $\ab$ we have
\[
   \Bigg\|F(b)\,g(b)-\sum_{j=1}^{\nu(P)}F(\tau_j)[g(\alpha_{j})-g(\alpha_{j{-}1})]\Bigg\|_X\le \|F(a)\,g(a)\|_X+\|g\|_{\infty}\,\SV_a^b(F),
\]
Furthermore, if $\int_a^bF\,\dd[g]$ exists then 
\begin{equation}\label{ineq2}
   \Big\|F(b)\,g(b)-\int_a^b F\,\dd[g]\Big\|_X\le \|F(a)\,g(a)\|_X+\|g\|_{\infty}\,\SV_a^b(F).
\end{equation}
\end{lemma}

\smallskip

Now we present the main result of this section. In what follows, $S_L(\ab,X)$ denotes the set of all finite step 
functions $g:\ab\to X$ which are left-continuous on $(a,b]$ and such that $g(a)=0$.

\smallskip

\begin{theorem}\label{SV-theo}
If $F\in SV(\ab,L(X))$, then
\[
\SV_a^b(F)=\sup\left\{\Big\|F(b)\,g(b)-\int_a^bF\,\dd[g]\Big\|_X\,;\,g\in S_L(\ab,X),\,\|g\|_\infty\leq 1 \right\}.
\]
\end{theorem}
\begin{proof}
At first, note that by \eqref{ineq2} $\SV_a^b(F)$ is an upper bound to the set
\[
\mathcal A:=
\left\{\Big\|F(b)\,g(b)-\int_a^bF\,\dd[g]\Big\|_X\,;\,g\in S_L(\ab,X),\,g(a)=0\mbox{ \ and \ }\|g\|_\infty\leq 1 \right\}.
\]
To conclude the proof it is enough to show that $\SV_a^b(F)\leq\sup \mathcal A$.

Let $\varepsilon>0$ be given. Then, there exist a division $D\,{=}\,\{\alpha_0,\alpha_1,\dots,\alpha_m\}$ of $\ab$ 
and $x_j\in X$, $j=1,\ldots,m$  
with $\|x_j\|\leq 1$ such that
\[
\SV_a^b (F)- \varepsilon<\Big\|\sum_{j=1}^m[F(\alpha_j)\,{-}\,F(\alpha_{j-1})]\,x_j\Big\|_X.
\]
Let $\tilde{g}:\ab\to X$ be the function given by 
\[
\tilde{g}(t)=\sum_{j=1}^m\chi_{(\alpha_{j-1},\alpha_j]}(t)\,x_j\quad\mbox{for \ }t\in\ab.
\]
Thus, $\tilde{g}$ is a left continuous step function with $\tilde{g}(a)=0$ and $\|\tilde{g}\|_\infty\leq 1$, that is, $\tilde{g}\in S_L(\ab,X)$. 
Calculating the integral $\int_a^bF\,\dd[\tilde{g}]$, we have
\[
\int_a^bF\,\dd[\tilde{g}]
=-\sum_{j=1}^{m-1}[F(\alpha_j)-F(\alpha_{j-1})]\,x_j+F(\alpha_{m-1})\,x_m
\]
(see \cite[Proposition 14]{Sch1}). Therefore,
\[
\SV_a^b (F)- \varepsilon<\Big\|\sum_{j=1}^m[F(\alpha_j)\,{-}\,F(\alpha_{j-1})]\,x_j\Big\|_X=\Big\|F(b)\,\tilde{g}(b)- \int_a^bF\,\dd[\tilde{g}]\Big\|_X\leq\sup \mathcal A.
\]
Since $\varepsilon>0$ is arbitrary, the result follows.
\end{proof}

\appendix\section{Appendix: Series in Banach space}

\begin{definition}
Let $x_n\in X$ for $n\in\N$. We say that:
\begin{enumerate}
\item The series $\dis\sum_{n=1}^\infty x_n$ is convergent if the sequence of its partial sums 
$s_n=\dis\sum_{k=1}^nx_k$ converges in $X$. 
\item The series $\dis\sum_{n=1}^\infty x_n$ is absolutely convergent if 
$\dis\sum_{n=1}^\infty\|x_n\|_X<\infty$. 
\item The series $\dis\sum_{n=1}^\infty x_n$ is unconditionally convergent if the series $\dis\sum_{n=1}^\infty x_{\pi(n)}$ converges in $X$ 
for any permutation $\pi$ of $\N$.
\end{enumerate}
\end{definition}

\smallskip

\begin{theorem}
If $\dis\sum_{n=1}^\infty x_n$ is unconditionally convergent, then all rearrangements have the same sum.
\end{theorem}
\noindent(See \cite[Theorem 1.3.1]{KK})

\smallskip

\begin{theorem}\label{T-1.3.2}
For series $\dis\sum_{n=1}^\infty x_n$ in $X$ the following conditions are equivalent:
\begin{enumerate}[(a)]
\item the series is unconditionally convergent;
\item for any bounded sequence $\{\alpha_n\}_n$ in $\R$, the series $\dis\sum_{n=1}^\infty \alpha_n x_n$ converges in $X$.
\end{enumerate}
\end{theorem}
\noindent(See \cite[Theorem 1.3.2]{KK} or \cite[Proposition 2.4.9]{Al})

\smallskip

\begin{lemma}\label{Thorp-lemma}
If the series $\dis\sum_{n=1}^\infty x_n$ is uncondittionally convergent in $X$, then
\[
\left\{\sum_{n=1}^\infty \alpha_n\,x_n\,:\,\alpha_n\in\R\mbox{ \ with \ }|\alpha_n|\leq 1,\,n\in\N\right\}
\]
is a bounded subset of $X$.
\end{lemma}
\noindent(See \cite[Lemma 1]{Thorp})

\medskip

It is clear that absolute convergence implies unconditional convergence. However, the converse is not true in general.

\smallskip

\begin{theorem}{\rm (Dvoretzky-Rogers)}\label{Theo-DR}
If every unconditionally convergent series in an Banach space $X$ is absolutely convergent, then the dimension of $X$ is finite.
\end{theorem}
\noindent(See \cite{DR}, \cite[Theorem 8.2.14]{Al} or \cite[Chapter VI]{D})

\begin{corollary}\label{D-R}
In every infinite-dimensional Banach space 
there exists an unconditionally convergent series that is not absolutely convergent.
\end{corollary}

\smallskip

In the sequel we recall some aspects of convergence of series involving weak topology.

\smallskip

\begin{definition}
Let $x_n\in X$ for $n\in\N$. We say that:
\begin{enumerate}
\item The sequence $\{x_n\}_n$ is a weakly Cauchy sequence if  the sequence $\{x^*(x_n)\}_n$ converges in $\R$ for every $x^*\in X^*$
\item The series $\dis\sum_{n=1}^\infty x_n$ is weakly convergent if there exists $z\in X$ such that  
the series $\dis\sum_{n=1}^\infty x^*(x_n)$ converges to $x^*(z)$ for every $x^*\in X^*$.
\item The series $\dis\sum_{n=1}^\infty x_n$ is weakly absolutely convergent if $\dis\sum_{n=1}^\infty |x^*(x_n)|<\infty$ 
for every $x^*\in X^*$.
\end{enumerate}
\end{definition}

\smallskip

\begin{proposition}\label{UC-wAC}
If the series $\dis\sum_{n=1}^\infty x_n$ is unconditionally convergent, then it is weakly absolutely convergent.
\end{proposition}
\noindent(See \cite[Proposition 2.4.4 $(iii)$]{Al})

\medskip

The converse of Proposition \ref{UC-wAC} charaterizes in an important class of Banach spaces.

\smallskip

\begin{theorem} {\rm (Bessaga-Pelczynski)}\label{Theo-BP}
A Banach space $X$ does not contain an isomorphic copy of $c_0$ if and only if 
every weakly absolutely convergent series in $X$ is unconditionally convergent.
\end{theorem}
\noindent(See \cite[Theorem V.8]{D}, \cite[Theorem 6.4.3]{KK} or \cite[Theorem 2.4.11]{Al})

\medskip

Another important class of spaces which is worth mentioning is the class of weakly sequentially complete Banach spaces.  
Recall that $X$ is weakly sequentially complete if every weakly Cauchy sequence is weakly convergent in $X$. 
In which concerns series in such spaces, we have the following result. 

\smallskip

\begin{theorem}\label{wsc-series}
If $X$ is weakly sequentially complete, then every weakly unconditionally convergent series is unconditionally convergent in $X$.  
\end{theorem}
\noindent(See \cite[Theorem 3.2.3]{HP} or \cite[Corollary 2.4.15]{Al})


\end{document}